\documentclass{amsproc}
\usepackage{hyperref}
\usepackage{euscript} 
\usepackage{mathrsfs}
\usepackage{bbm}
\usepackage{amssymb} 
\usepackage{amsfonts,amsmath,amsxtra,mathdots,mathabx}
\usepackage{color}
\usepackage{tikz}
\usepackage{appendix,upgreek}

\allowdisplaybreaks

\DeclareFontFamily{U}{matha}{\hyphenchar\font45}
\DeclareFontShape{U}{matha}{m}{n}{
	<5> <6> <7> <8> <9> <10> gen * matha
	<10.95> matha10 <12> <14.4> <17.28> <20.74> <24.88> matha12
}{}
\DeclareSymbolFont{matha}{U}{matha}{m}{n}

\DeclareMathSymbol{\Lt}{3}{matha}{"CE}
\DeclareMathSymbol{\Gt}{3}{matha}{"CF}

\DeclareSymbolFont{mathc}{OML}{txmi}{m}{it}% txfonts
\DeclareMathSymbol{\varvv}{\mathord}{mathc}{118}
\DeclareMathSymbol{\varww}{\mathord}{mathc}{119}
\DeclareMathSymbol{\vnu}{\mathord}{mathc}{"17}

\DeclareSymbolFont{mathd}{OML}{ztmcm}{m}{it}
\DeclareMathSymbol{\varalpha}{\mathord}{mathd}{11}
\DeclareMathSymbol{\varlambda}{\mathord}{mathd}{21}

\DeclareMathSymbol{\depsilon}{\mathord}{mathd}{15}

\def\vepsilon{\upvarepsilon} 

\DeclareMathSymbol{\varchi}{\mathord}{mathd}{31}
 
\def\valpha{\text{\scalebox{0.88}[1.02]{$\alpha$}}}

\newcommand{\BC}{{\mathbb {C}}}

\newcommand{\BQ}{{\mathbb {Q}}} \newcommand{\BR}{{\mathbb {R}}}

 \newcommand{\BZ}{{\mathbb {Z}}}

\newcommand{\RC}{{\mathrm {C}}}

\newcommand{\RI}{{\mathrm {I}}} 
 
 \newcommand{\RN}{{\mathrm {N}}}
 
 \newcommand{\RR}{{\mathrm {R}}}

\newcommand{\GL}{{\mathrm {GL}}}
\newcommand{\PGL}{{\mathrm {PGL}}}

\newcommand{\sstyle}{\scriptstyle}
\newcommand{\ssstyle}{\scriptscriptstyle}

\newcommand{\ra}{\rightarrow} 
\def\viint{	\int \hskip -4 pt \int}

\def\fra{\mathfrak{a}}
\def\frm{\mathfrak{m}}
\def\frn{\mathfrak{n}}

\def\-{^{-1}}

\def\sasymp{\text{ \small $\asymp$ }}
\def\mod{\mathrm{mod}\,  }

\def\sumh{\sideset{}{^h}\sum}

\def\nd{\mathrm{d}}
\def\Tr{\mathrm{Tr}}

\def\Fx{F^{\times}}
\def\trh{ \mathrm{trh}}

\def\lp {\left (}
\def\rp {\right )}

\def\sstimes {\scalebox{0.55}{$\times$}}

\def\boldJ {\boldsymbol J}
\renewcommand{\Im}{{\mathrm{Im} }}

\def\shskip{\hskip 0.5 pt}

\def\frO {\text{\raisebox{- 2 \depth}{\scalebox{1.1}{$ \text{\usefont{U}{BOONDOX-calo}{m}{n}O} \hskip 0.5pt $}}}}
\def\frOO {\text{\raisebox{- 2 \depth}{\scalebox{1.1}{$ \text{\usefont{U}{BOONDOX-calo}{m}{n}O}$}}}}
\def\frp{\mathfrak{p}}

\def\frb{\mathfrak{b}}
\def\frc{\mathfrak{c}}
\def\frd{\mathfrak{d}}

\def\frD{\mathfrak{D}}

\def\SB{\text{\raisebox{- 2 \depth}{\scalebox{1.1}{$ \text{\usefont{U}{BOONDOX-calo}{m}{n}B} \hskip 0.5pt $}}}}

\def\SD{\text{\raisebox{- 2 \depth}{\scalebox{1.1}{$ \text{\usefont{U}{BOONDOX-calo}{m}{n}D} \hskip 0.5pt $}}}}

\def\SE{\text{\raisebox{- 2 \depth}{\scalebox{1.1}{$ \text{\usefont{U}{BOONDOX-calo}{m}{n}E} \hskip 0.5pt $}}}}

\def\SO{\text{\raisebox{- 2 \depth}{\scalebox{1.1}{$ \text{\usefont{U}{BOONDOX-calo}{m}{n}O} \hskip 0.5pt $}}}}

\def\SM{\text{\raisebox{- 2 \depth}{\scalebox{1.1}{$ \text{\usefont{U}{BOONDOX-calo}{m}{n}M} \hskip 0.5pt $}}}}

\def\SN{\text{\raisebox{- 2 \depth}{\scalebox{1.1}{$ \text{\usefont{U}{BOONDOX-calo}{m}{n}N} \hskip 0.5pt $}}}}

\def\SR{\text{\raisebox{- 2 \depth}{\scalebox{1.1}{$ \text{\usefont{U}{BOONDOX-calo}{m}{n}R} \hskip 0.5pt $}}}}

\def\SZ{\text{\raisebox{- 2 \depth}{\scalebox{1.1}{$ \text{\usefont{U}{BOONDOX-calo}{m}{n}Z}\hskip 1pt $}}}}

\def\SDH{\text{\raisebox{- 1 \depth}{\scalebox{1.06}{$ \text{\usefont{U}{dutchcal}{m}{n}H}  $}}}}

\newcommand{\delete}[1]{}

\theoremstyle{plain}

\newtheorem{thm}{Theorem}[section] \newtheorem{cor}[thm]{Corollary}
\newtheorem{lem}[thm]{Lemma}  \newtheorem{prop}[thm]{Proposition}

 \newtheorem{defn}[thm]{Definition}

\newtheorem {rem}[thm]{Remark}

\newtheorem*{acknowledgement}{Acknowledgements}

\numberwithin{equation}{section}

\begin{document}
	
\title[Moments of Central $L$-values for Maass Forms]{{Moments of Central $L$-values for Maass Forms over Imaginary Quadratic Fields}}

\author{Sheng-Chi Liu}%
\address{Department of Mathematics and Statistics, Washington State University,
	Pullman, WA 99164-3113, USA}%
\email{scliu@math.wsu.edu}%

\author{Zhi Qi}
\address{School of Mathematical Sciences\\ Zhejiang University\\Hangzhou, 310027\\China}
\email{zhi.qi@zju.edu.cn}

\thanks{The first author was supported by a grant (\#344139) from the Simons Foundation. The second author was supported by a grant (\#12071420) from the National Natural Science Foundation of China.}


\begin{abstract}
	In this paper, over imaginary quadratic fields, we consider the family of $L$-functions $L (s, f)$ for an orthonormal basis of spherical Hecke--Maass forms $f$ with Archimedean parameter $t_f$. We establish asymptotic formulae for the twisted first and second moments of the central values $L\big(\frac 1 2, f\big)$, which can be applied to prove  that at least $33 \%$ of  $L\big(\frac 1 2, f\big)$ with $t_f \leqslant T$ are non-vanishing as $T \rightarrow \infty$. Our main tools are the spherical Kuznetsov trace formula and the Vorono\"i summation formula over imaginary quadratic fields. 
\end{abstract} 

	\maketitle

{\small \tableofcontents}

% % % % % % % % % % % % % % % % % % % % % % % % % % % % % % % % % % % % % % % % % % % % % % % %
\section{Introduction}

A recurring theme in analytic number theory is the study of central value  of a family of $L$-functions. In this paper, we prove asymptotic formulae for the  twisted first and second moments of central $L$-values for the family of Hecke--Maass cusp forms over the classical  modular group $\PGL_2 (\BZ)$ or a Bianchi modular group $\PGL_2 (\frO)$ (here $\frO$ is the ring of integers of a imaginary quadratic field $F$). As a standard application, we obtain non-vanishing results  for the central value of such Maass form $L$-functions in the Archimedean aspect. 

There are abundant non-vanishing results for holomorphic modular forms over $\BQ$ in  \cite{Duke-1995,IS-Siegel,KM-Analytic-Rank,KM-Analytic-Rank-2,VanderKam-Rank,KMV-Derivatives,Djankovic-Gamma1,Rouymi-1,Rouymi-2,BF-prime-power,Lau-Tsang-Mean-Square,Luo-Weight,BF-Moments,Liu-L-Derivative,Jobrack-Derivative} and also over a totally real field in \cite{Trotabas-Hilbert}.

Recently there are two papers \cite{SH-Liu-Maass} and \cite{BHS-Maass} on the non-vanishing of central $L$-value  for the family of Maass forms for $\PGL_2 (\BZ)$ in the aspect of spectral parameter $t_f$; %\footnote{It should be mentioned that a non-vanishing result for $L \big(\frac 1 2 + it_f, f \big)$ was obtained in \cite{Xu-Nonvanishing}.} 
in the former, the existence of a positive proportion of non-vanishing is proven for eigenvalues in short intervals, while in the latter, a lower bound for the proportion is  obtained effectively. In both works,  a formula of Motohashi\footnote{As indicated in \cite[\S 3.6]{Motohashi-Riemann}, this formula was claimed by Kuznetsov \cite{Kuznetsov-Motohashi-formula} with no
	rigorous proof. It should therefore be called the Kuznetsov--Motohashi formula.  Nevertheless, to avoid confusion, we shall still name it after Motohashi.} (see \cite[Lemma]{Motohashi-JNT-Mean} or \cite[Lemma 3.8]{Motohashi-Riemann}) is used for the twisted second moment, but the authors of \cite{BHS-Maass} are able to obtain an asymptotic formula so that their effective non-vanishing result becomes possible.

In this paper, we use the formula of Kuznetsov instead of Motohashi. The reader might wonder: ``What is new here?  The Kuznetsov formula has already been used for a lot of problems."  To illustrate the novelty of this work, we need to answer two questions:
\begin{itemize}
	\item [(1)] Why do the previous authors abandon the Kuznetsov formula?
	\item [(2)] Why do we abandon the Motohashi formula?
\end{itemize} 

There are two Kuznetsov formulae for $\mathrm{PSL}_2 (\BZ)$ (see \cite{Kuznetsov,Kuznetsov-Motohashi-formula} or \cite[Theorem 2.2, 2.4]{Motohashi-Riemann}):   weighted either with or without root number $\epsilon_f$ and containing either the $J$- or the $K$-Bessel function. The Kuznetsov formula for $\PGL_2 (\BZ)$ is deduced from summing up these two formulae (Maass forms for $\PGL_2 (\BZ)$ are termed even forms in the literature). 

In \cite[\S 3.3]{Motohashi-Riemann}, Motohashi derives his formula  from the Kuznetsov  formula for $\mathrm{PSL}_2 (\BZ)$ that is weighted by $\epsilon_f$ and contains    $K_{2it}(x)$. A simple but crucial observation  is that $L \big(\frac 1 2 , f \big) = 0$ if   $\epsilon_f = - 1$.  
He avoids using the other Kuznetsov formula with $ J_{2it}(x)$, as ``the relevant transformation   is difficult to handle because its integrand is not of rapid decay''.

In  \cite{SH-Liu-Maass}, Shenhui Liu follows Wenzhi Luo's mollification analysis for the holomorphic case in  \cite{Luo-Weight}. In his Introduction, he lists several advantages of Motohashi's formula for the mollified second moment and explains that there is a ``deeper reason for using Motohashi's formula'': If one were to use the Kuznetsov formula, it would not be easy ``to extract information from the off-diagonal terms by using properties of Estermann zeta-functions'',  ``since the Mellin--Barnes representation of $J_{2it}(x)$ gives very narrow room for contour shifting''.

 A direct approach by the Kuznetsov formula is certainly of its own interests and merits. A  more general goal   on our mind is to obtain a positive proportion for the non-vanishing problem over an imaginary quadratic field $F$. However, there is currently no Motohashi formula  over a field other than $\BQ$. Indeed, the root number $\epsilon_f$ is always $+1$ for any spherical Maass form $f$ over $\mathrm{PSL}_2 (\frO)$, so Motohashi's idea does not work here directly. At any rate,  one must reconsider  the problem over $\BQ$ and find a way to solve it without recourse to the Motohashi formula. 

 % are mainly from complex analysis. 
 Our idea is to bypass the   difficulties encountered in \cite[\S 3.3]{Motohashi-Riemann} and \cite{SH-Liu-Maass} by 
\begin{itemize}
	\item [(1)] applying the Vorono\"i summation as a substitute of the functional equation for Estermann zeta-functions, and
	\item [(2)] applying the Fourier-type representation  instead of the Mellin--Barnes representation  for   Bessel functions. 
\end{itemize}
The Vorono\"i summation has occurred in the case of holomorphic modular forms in \cite{Hough-Zero-Density,BF-Moments}, but their analyses are quite different from ours.   A key feature of our analysis is a uniform treatment of the integrals involving $J_{2it} (x)$ and $K_{2it} (x)$ which can be extended to the complex setting. More explicitly, we shall have   Fourier integrals in the off-diagonal terms, and the problem will be reduced to estimating the area of certain regions defined via hyperbolic or trigonometric-hyperbolic functions.

\subsection*{Statement of Results}

Let $F  = \BQ  $   or an imaginary quadratic field $\BQ (\sqrt{d_F})$ of discriminant $d_F$ and class number $h_F = 1$. Let $\frO $ be its ring of integers. Let $N  = 1$ or $2$ be the degree of $F$. For a nonzero integral ideal $\frn \subset \frO$ let $\RN (\frn) = |\frO / \frn|$ be its norm and $\tau (\frn) = \sum_{\frd | \frn} 1$.

Let $\gamma_{0}$ and $\gamma_{1}$ respectively be the       constant term and the residue  of Dedekind's $\zeta_F (s)$  at  $s=1$. Define $ c_1 = 1/8\pi^2 $ or $\sqrt{|d_F|}/8\pi^3$ according    as $F = \BQ$ or $\BQ (\sqrt{d_F})$. 
Moreover, let  $\theta = \frac {13}{84}$,  $\theta' = \frac {9} {56}$ (see \S \ref{sec: zeta function}), and for $0 < \beta \leqslant 1$ define
\begin{align}\label{9eq: exponent for E2} 
\valpha_1 = \left\{ \begin{aligned}
	& \hskip -1.5pt 0, &&     \text{if }  \tfrac {1273} {4053}+\vepsilon  \leqslant \beta \leqslant 1, \\
	& \hskip -1.5pt 2 \theta ,   &&    \text{if }  0 < \beta <  \tfrac {1273} {4053} +\vepsilon ,
\end{aligned} \right. \hskip 8pt	\valpha_2 = \left\{ \begin{aligned}
		& \hskip -1.5pt 0, &&     \text{if }  \tfrac 2 3+\vepsilon  \leqslant \beta \leqslant 1, \\
		& \hskip -1.5pt 2 \theta , &&   \text{if }  \tfrac {1273} {4053}+\vepsilon  \leqslant \beta <  \tfrac 2 3+\vepsilon , \\
		& \hskip -1.5pt 4 \theta, &&    \text{if }  0 < \beta <  \tfrac {1273} {4053} +\vepsilon ,
	\end{aligned} \right.  
\end{align}
  if $F = \BQ$, and
\begin{align}
\valpha_1 = \left\{ \begin{aligned}
	& \hskip -1.5pt 0, &&    \text{if }  \tfrac 7 8+\vepsilon  \leqslant \beta \leqslant 1, \\ 
	& \hskip -1.5pt 2 \theta', &&   \text{if }  0 < \beta <  \tfrac {7} {8} +\vepsilon ,
\end{aligned} \right. \quad 	\valpha_2 = \left\{ \begin{aligned}
		& \hskip -1.5pt 2 \theta', &&   \text{if }  \tfrac 7 8+\vepsilon  \leqslant \beta \leqslant 1, \\ 
		& \hskip -1.5pt 4 \theta', &&    \text{if }  0 < \beta <  \tfrac {7} {8} +\vepsilon ,
	\end{aligned} \right.  
\end{align}
 if $F = \BQ (\sqrt{d_F})$.

Let $ \SB $ be an orthonormal basis consisting of Hecke--Maass cusp forms for the spherical cuspidal spectrum for $ \PGL_2 (\frO)$. For $f \in \SB$, let  $ t_{f } \in [0, \infty) \cup i \big[ 0, \frac 1 2 \big)$ be its Archimedean parameter, $\lambda_f (\frn)$ be its Hecke eigenvalues, $L(s, f)$ and $L(s, \mathrm{Sym}^2 f)$ respectively be its standard and symmetric square $L$-functions. For any sequence of complex numbers $a_f$ we introduce the harmonic summation
\begin{align}
	\sumh_{f \in  \SB} a_f = \sum_{f \in  \SB} \omega_f a_f, \qquad \omega_f = \frac 1 {2 L(1, \mathrm{Sym}^2 f)} . 
\end{align} 

For large  $T > M   $    we define  
\begin{align}\label{1eq: defn of kq}  
	k  (t) =    e^{- (t  - T)^2 / M^2} + e^{-(t + T)^2 / M^2}     ,  
\end{align}
and for $q = 1$ or $2$ we introduce the  smoothly weighted twisted moments: 
\begin{align}
	\SM_q (\frm ) =  \sumh_{f \in  \SB  } \hskip -1pt    k   ( t_f ) \lambda_f  ( \frm )    L \big( \tfrac 1 2 , f \big)^q  .   
\end{align}

\begin{thm}\label{thm: moment}
%	Let notation be as above.   
Define 
%	\begin{align}\label{10eq: defn of gamma}
	$	\gamma_0 '  = \gamma_{0} - \gamma_{1}   \log \big( (2\pi)^N/|d_F| \big) $.  
%	\end{align}  
Let $ M = T^{\beta}$ with $\vepsilon \leqslant \beta \leqslant 1-\vepsilon$.  Then
	\begin{align}
		\label{10eq: 1st moment}
		\SM_1 (\frm ) =  4 \sqrt{\pi} c_1 \frac { M T^{N}    } {\sqrt{\RN(\frm)}}   \big(  1 + O_{\vepsilon}  \big( (M/T)^2 \big)  \big) + O_{\vepsilon} \lp M T^{   N   \valpha_1 + \vepsilon}  \rp, 
	\end{align}
	for $\RN (\frm) \leqslant T^{N-\vepsilon}$, and %there are constants $\gamma_{1}$ and $e_F$ such that 
	\begin{align}
		\label{10eq: 2nd moment}
		\begin{aligned}
			\SM_2 (\frm ) =  8 \sqrt{\pi} c_1 \frac { \tau (\frm) M T^N   } {\sqrt{\RN(\frm)}}    \bigg(   \gamma_{1}  \log \frac {T^N} {\sqrt{\RN(\frm)}} + \gamma_0 '   + O_{\vepsilon}  ( M T^{\vepsilon} /T )  \bigg) & \\ +     O_{\vepsilon}    \bigg(  M T^{   N   \valpha_2  + \vepsilon} + \frac {\sqrt{\RN(\frm) } T^{N/2 + \vepsilon}} {M^{1-N/2}}  \bigg) & , 
		\end{aligned}
	\end{align}
	for   $\RN (\frm) \leqslant T^{2 N -\vepsilon}$.  
	Moreover, if $F = \BQ$, then the error term $O_{\vepsilon}  ( M T^{\vepsilon} /T )$ in the first line of {\rm\eqref{10eq: 2nd moment}} may be improved into $O_{\vepsilon}    \big(  (M/T)^2 \log T \big) $ and the second error term in the second line may be removed if $ \RN (\frm) \leqslant M^{2-\vepsilon}$. 
\end{thm}

%\begin{rem}\label{rem: m < T}
	The error terms are always inferior to the main term in {\rm\eqref{10eq: 1st moment}}  as long as $\RN(\frm) \leqslant T^{N-\vepsilon}$, while this holds for {\rm\eqref{10eq: 2nd moment}} as long as 
\begin{align}\label{1eq: m < T}
	\RN (\frm) \leqslant \min \left\{ M^{2-N/2} T^{N/2 - \vepsilon}, T^{2 N (1 - \valpha_2) - \vepsilon} \right\} . 
\end{align}
%\end{rem}

When $F = \BQ$, with  cleaner error terms, \eqref{10eq: 1st moment} and  \eqref{10eq: 2nd moment} are essentially Theorem 4.1 and 5.1 in \cite{BHS-Maass} prior to the   averaging process for the $T$-parameter.  

%We have a simpler asymptotic formula on setting $\frm = (1)$. 

For $T^{\vepsilon} \leqslant 3 H \leqslant T $ define
\begin{align}\label{1eq: truncated moments}
\SN_q  (T, H) = 	\sumh_{|t_f-T| \leqslant H} L \big( \tfrac 1 2 , f \big)^q + \frac {1}  {2\pi} \gamma_1 \int_{\, T-H}^{T+H}        \frac {\left| \zeta_F \big(\tfrac 1 2 + it \big) \right|^{2q} } { | \zeta_F  (1 + 2 it  )  |^2 }
  \shskip   \nd t .
\end{align} 
%By the   average process of Ivi\'c and Jutila, 
Then we have  simpler asymptotic formulae for   $\SN_q (T, H)$ as follows. %The proof is by  an  averaging process  for the case $\frm = (1)$. 
%For the fisrt moment it is critical that $ L \big(\frac 1 2 , f\big) $ is non-negative   (\cite{Guo-Positive}). 

\begin{cor} \label{cor: unsmooth} 
 	For $F = \BQ$ we have
\begin{align}\label{1eq: N1, Q}
	\SN_1 (T, H) =  \frac 1 {\pi^2}  H T    + O_{\vepsilon}  \lp T^{1+\vepsilon} \rp, 
\end{align}
and
\begin{align}\label{1eq: N2, Q}
\SN_2  (T, H) = \frac 1 {\pi^2} \int_{\,T-H}^{T+H}  K      (       \log  K   + \gamma_0 '    ) \nd K %\frac {1} {2 \pi^2} \hskip -1pt \sum_{\pm} \hskip -1pt \pm (T \hskip -1pt \pm \hskip -1pt H)^2 \log (T \hskip -1pt \pm \hskip -1pt H) \hskip -1pt + \hskip -1pt \frac  {2 \gamma_0'   - 1} {\pi^2} HT 
 +   O_{\vepsilon}  \big( T^{1+\vepsilon}    \big) .
\end{align}  
For $F = \BQ (\sqrt{d_F})$ we have 
\begin{align}
	\SN_1 (T, H) =  \frac {\sqrt{|d_F|}} {3 \pi^{3}  } \big( 3 H T^2 + H^3 \big)     + O_{\vepsilon}  \lp T^{2+\vepsilon} \rp, 
\end{align}
and 
\begin{align}
	\SN_2  (T, H) = \frac {\sqrt{|d_F|}} {  \pi^{3}  } \int_{\,T-H}^{T+H} K^2         (  2 \gamma_{1}   \log  K   + \gamma_0 '    ) \nd K + O_{\vepsilon} \big(T^{2+\vepsilon}\big).
\end{align} 
\end{cor}

  In the case $ H = T/3$,   the formulae \eqref{1eq: N1, Q} and \eqref{1eq: N2, Q} should be compared with \cite[Theorem 1]{Iviv-Jutila-Moments} and  \cite[Theorem 2]{Motohashi-JNT-Mean}.

As a consequence of the mollification technique as in \cite{IS-Siegel,KMV-Derivatives}, one may derive  from the asymptotic formulae in Theorem \ref{thm: moment} the following  effective lower bound for the  proportion of  non-vanishing $ L \big(\frac 1 2 , f\big) $. 

\begin{thm}\label{thm: non-vanishing}
	For any $\vepsilon > 0$ and sufficiently large $T$, we have 
	\begin{align}
		\mathop{\sumh_{|t_f - T| \leqslant H}}_{L (\frac 1 2 , f) \neq 0} 1  \geqslant \lp \frac {\varDelta } {1 + \varDelta } -\vepsilon \rp \sumh_{|t_f - T| \leqslant H} 1, 
	\end{align}
where $3 H = T^{\beta}$ with $\vepsilon \leqslant \beta \leqslant 1$ and  \begin{align}\label{1eq: Delta}
	\varDelta \leqslant \min \bigg\{ 1 - \valpha_2 ,   \frac 1 4 + \lp \frac 1 {N} - \frac 1 {4} \rp \beta \bigg\}.  
\end{align}
\end{thm}

For $F = \BQ$   this is essentially Theorem 1.2 in \cite{BHS-Maass}. %For $F = \BQ (\sqrt{d_F})$ the proof is literally the same. 
For  $F = \BQ(\sqrt{d_F})$ it follows from almost the same arguments in \cite[\S 8]{BF-Moments} and \cite[\S 7]{BHS-Maass}. As such, we omit the details of proof and only remark   that  the limitation   \eqref{1eq: Delta} comes from the inequality   \eqref{1eq: m < T}. To avoid extra work on $ L(1, \mathrm{Sym}^2 f)$,  we allow the  harmonic weight  to be present (see the paragraph below (2.9) in \cite{IS-Siegel}). 

The following results follow  if we choose $H = T/3$  in Theorem \ref{thm: non-vanishing} and use a dyadic partition. 

\begin{cor}    We have
	\begin{align}
		\mathop{\sumh_{ t_f   \leqslant T}}_{L (\frac 1 2 , f) \neq 0} 1 \geqslant \lp \frac 1 2 -\vepsilon \rp \sumh_{t_f   \leqslant T } 1  
	\end{align}
if $F = \BQ  $, and  
	 \begin{align}
	 	\mathop{\sumh_{ t_f   \leqslant T}}_{L (\frac 1 2 , f) \neq 0} 1 \geqslant \lp \frac 1 3 -\vepsilon \rp \sumh_{t_f   \leqslant T } 1  
	 \end{align}
if $F = \BQ (\sqrt{d_F})$. 
\end{cor}

%The $\frac 13$ above is somehow weaker than the $\frac 12$ in the case $F = \BQ$ (see Theorem 1.1 in \cite{BHS-Maass}).

%Finally we remark that Theorem \ref{thm: non-vanishing}  yields at least $20 \%$ non-vanishing   central $L$-values as long as $H \geqslant T^{\vepsilon}$, which is comparable to \cite[Theorem 1.1]{BF-Moments} in the holomorphic case.

%\begin{cor}\label{cor: non-vanishing, 2}
%	For any $\vepsilon > 0$ and sufficiently large $T$, we have 
%	\begin{align}
%		\mathop{\sumh_{|t_f - T| \leqslant T^{\vepsilon}}}_{L (\frac 1 2 , f) \neq 0} 1  \geqslant \lp \frac 1 5 -\vepsilon \rp \sumh_{|t_f - T| \leqslant H} 1. 
%	\end{align} 
%\end{cor}
 
Finally, we remark that, with some efforts, our results can be extended to an arbitrary imaginary quadratic field (see \cite{Qi-GL(3)}). %With more careful analysis, the conditions $M, H \geqslant T^{\vepsilon}$ may be improved into $ M \geqslant C \sqrt{\log T} $ (for $ C $ large) and $H \geqslant \log T$ (see \cite{Iviv-Jutila-Moments}). 

\subsection*{Notation} By $X \Lt Y$ or $X = O (Y)$ we mean that $|X| \leqslant c Y$ for some constant $c > 0$, and by $X \asymp Y$  we mean that  $X \Lt Y$ and $Y \Lt X$. We write $X \Lt_{P, \shskip Q, \, \dots} Y$ or $X = O_{P, \shskip Q, \, \dots} (Y)$ if the implied constant $c$ depends on  $P$, $Q, \dots$. We say  that $X$ is negligibly small if $X = O_A (T^{-A})$ for arbitrarily large but fixed $A \geqslant 0$.

We adopt the usual $\vepsilon$-convention of analytic number theory; the value of $\vepsilon $ may differ from one occurrence to another.

\begin{acknowledgement}
	 We thank the referee for careful readings and  helpful	comments.
\end{acknowledgement}

 { \large \part{Preliminaries}}

\section{Number Theoretic Notation}\label{sec: notation}

%We give a list of our most frequently used notation from algebraic number theory. %We follow Lang \cite{Lang-ANT} closely.

\subsection{Basic Notions}
Let $F  = \BQ  $   or an imaginary quadratic field $\BQ (\sqrt{d_F})$ of class number $h_F = 1$, where $d_F$ is the discriminant of $F$. Let $N = 1$ or $2$ be the degree of $F$.  Let $\frO $ be its ring of integers and $\frOO^{\times}$  be the group of units.  Let $\mathfrak{D}$ be the different ideal of $F$ and $\frO' = \frD^{-1}$ be the dual of $\frO$. Let $w_F$ be the number of roots of unity in $F$. 
Let $\mathrm{N}$ and $\Tr$ denote the norm and the trace for $F    $, respectively.

Let $F_{\infty}$ be the Archimedean completion of $F$.  Let $\| \hskip 3.5 pt \|_{\infty} = | \hskip 3.5 pt |^N$  denote the normalized module of $F_{\infty}$, where $| \hskip 3.5 pt |$ is the usual absolute value. Define the additive character $\psi_{\infty} (x) = e (- x)$ if $F    _\infty = \BR$ and  $\psi_{\infty} (z) = e (- (z + \widebar z))$ if $F    _\infty = \BC$.  
We choose the Haar measure $\nd x$   of $F    _{\infty}$ self-dual with respect to $\psi_{\infty}$: the Haar measure is the ordinary Lebesgue measure on the real line if $F    _{\infty} = \BR$, and twice  the ordinary Lebesgue measure on the complex plane if $F    _{\infty} = \BC$. % The Fourier transform of a Schwartz function $f (x)$ on $F_{\infty}$ is defined by 
%\begin{align*}%\label{1eq: defn Fourier}
%	\widehat{f} (y) = \int_{F_{\infty}} f(x) \psi_{\infty} (xy) \nd x. 
%\end{align*}

In general, we use Gothic letters $\fra ,  \frb , \frm, \frn, \dots$ to denote {\it nonzero} integral ideals of $F$.   Let $\frp$ always stand for a prime ideal. 
Let $\RN (\fra)$ denote the norm of $\fra$.

\subsection{Arithmetic Functions}  

Let $\tau (\frn)  $ and $\mu (\frn)$  be the  divisor function and the M\"obius function. %, and the Euler totient  function for $F$. More explicitly, $ \tau (\frn) = \sum_{\frd |\frn} 1 $, $ \mu  (\frn) = (-1)^k $ if $\frn$ is square-free with $k$ distinct prime divisor ideals, $\mu (\frn) = 0$ if $\frn$ is not square-free, and $\varphi (\frn) = \RN (\frn) \prod_{\frp |\frn} (1- \RN (\frp)^{-1})$. 
For $s \in \BC$ define
\begin{align}\label{1eq: defn of tau (n)}
	\tau_s (\frn ) = \tau_{-s}  (\frn ) =  \sum_{ \sstyle \fra \frb = \frn  }  \RN \big(\fra \frb^{-1} \big)^{ s }. 
\end{align}

\subsection{Kloosterman and Ramanujan Sums} 
For $m, n \in \frO'$  and $c \in \frO$ we define
\begin{align}\label{2eq: defn Kloosterman KS}
	S (m, n ; c ) = \sum_{a \, \in (\frO /c)^{\sstimes} } \psi_{\infty} \bigg( \frac {m a + n \widebar{a} } {c} \bigg), 
\end{align} 
where $\widebar{a} $ is the multiplicative inverse of $a$ modulo $c$.   
The sum $S(m, 0; c)$ is usually named after Ramanujan. We have 
\begin{align}\label{2eq: Ramanujan}
	S(m, 0; c) = \sum_{\frd | (m \frD, c)} \RN(\frd) \mu \big( c \frd^{-1} \big) . 
\end{align}

\subsection{The Dedekind Zeta Function}\label{sec: zeta function}

Let $\zeta_F (s)$ be the Dedekind   $\zeta$ function for $F$:
\begin{align}
	 \zeta_F (s) = \sum_{\frn \shskip \subset \shskip \frO }  \frac { 1 } {\RN(\frn)^{  s} }, \qquad \mathrm{Re} (s) > 1. 
\end{align}
It is well-known that  $\zeta_F (s)$ is a meromorphic function on the complex plane with a simple pole at $s=1$. %Let   
Define the constants $\gamma_{0} $ and  $\gamma_{1} $ by % respectively be the  the constant term and the residue  of  $\zeta_F (s)$  at  $s=1$; namely, 
\begin{align}\label{2eq: zeta (s), s=1}
	\zeta_F (s) = \frac {\gamma_{1}} {s-1} + \gamma_{0} + O (|s-1|), \qquad s \ra 1. 
\end{align}  
For $F = \BQ (\sqrt{d_F})$ we have $\zeta_F (s) = \zeta (s) L (s, \chi^{}_{d_F})$ with $\chi^{}_{d_F}$ the primitive quadratic character associated to $F$.

Let $\theta > 0$ be a  sub-convex exponent for $ \zeta_F (s) $; namely, 
\begin{align}\label{2eq: subconvex}
	\zeta_F  \big( \tfrac 1   2 + it  \big) \Lt_{\vepsilon } (1+|t|)^{N \theta + \vepsilon}
\end{align}
for any $\vepsilon > 0$. %For arbitrary $F$ the Weyl exponent $\theta =   \frac 1   6$ is admissible \cite{Heath-Brown-Weyl}, while f
For the Riemann $\zeta (s)$  the best sub-convex exponent  to date   $\theta = \frac  {13}  {84}$ is due to Bourgain \cite{Bourgain}. This together with the Weyl sub-convex bound for $ L (s, \chi^{}_{d_F}) $ (see for example \cite{H-B-Hybrid}) yields $\theta = \frac {9} {56}$ in the case $F = \BQ (\sqrt{d_F})$. It should be remarked that for arbitrary  $F$ the Weyl exponent $\theta =   \frac 1   6$ is always admissible \cite{Heath-Brown-Weyl}.

\section{\texorpdfstring{Automorphic Forms on $\GL_2$}{Automorphic Forms on GL(2)}}

%We shall only consider . 
In this section, we briefly compile some results and introduce the relevant notation  from the theory of spherical automorphic forms on $\PGL_2 (\frO) \backslash \PGL_2 (F_{\infty}) / K_{\infty}$,  where $K_{\infty}  = \mathrm{O}_2 (\BR) / \{\pm 1_{2} \}$ or $  \mathrm{U}_2 (\BC)/  \{\pm 1_{2} \}$ according as   $F_{\infty} = \BR $ or $ \BC$, especially the Kuznetsov trace formula and the Vorono\"i summation formula (for Eisenstein series). The reader is referred to \cite{Qi-GL(3),Qi-VO,Venkatesh-BeyondEndoscopy} for further details.

%We denote by $Z$ the center of $\GL_2$. %Let $\PGL_2 = \GL_2 / Z$ as usual. 

%Let For each non-Archimedean $v$, let  $ K_{v} =  \PGL_2 (\frO_{v}) $. %Set $K (\frO) = \prod_{v \shskip \nmid \infty} K_{v} (\frO_{v})  $ ($\subset \GL_2 (\BA_{f})$). Note that $  \PGL_2 (F) \cap K_f = \PGL_2 (\frO)$ (the intersection is taken in $\GL_2 (\BA_f)$). %For each Archimedean place $v$, we fix a maximal compact subgroup $K_{v}$ of $\GL_2 (F_{v})$. 
 %By definition, an automorphic form  on $\PGL_2 (F) \backslash \PGL_2 (\BA)$ is spherical if it is invariant under $K$. 

%We identify $N (F_{\infty})$ with $F_{\infty}$ and $ A (F_{\infty}) $ with $F_{\infty}^{\times}$, and define their measures %on $N (F_{\infty})$ and $ A (F_{\infty}) $ accordingly. For $v | \infty$, we normalize the Haar measure on $K_{v}$ so that $K_{v} /  A (F_{v}) \cap K_{v}   $ has measure $1$.  Thus the measure of $K_{v}$  is $2$ or $2\pi$ according as $v$ is real or complex.  The Haar measure on $\PGL_2 (F_{\infty})$ is defined  via the Iwasawa decomposition $ \PGL_2 (F_{\infty}) = N (F_{\infty}) A(F_{\infty}) K_{\infty} $.

\vskip 5 pt

\subsection{Archimedean Representations}\label{sec: Archimedean} In this paper, we shall be concerned
only with {spherical} representations of $\PGL_2 (F_{\infty})$.   By definition, an irreducible representation of $\PGL_2 (F_{\infty})$ is spherical if it contains a nonzero $K_{\infty}$-invariant vector. 

Let $Y  = (-\infty, \infty) \cup i \hskip -1.5 pt \left(- \frac 1 2, \frac 1 2 \right)$.  We associate to $t \in Y$ a unique spherical unitary irreducible representation $\pi (i t)$ of $\PGL_2 (F_{\infty})$. Namely, the parameter $t $ determines a character of the  diagonal torus via
\begin{align*}
	\begin{pmatrix}
		x & \\ & y
	\end{pmatrix} \ra   \|x / y\|_{\infty}^{i t} , \qquad x, y \in F_{\infty}^{\times}, 
\end{align*}
and we let $\pi (i t)$ be the irreducible spherical constituent of the representation unitarily induced
from this character.     For $t$ real, the spherical $ \pi (i t) $ is   tempered, and the Plancherel measure $\nd   \mu (t)  $ is defined by
\begin{equation}\label{1eq: defn Plancherel measure}
	\nd  \mu (t) =
	\left\{ \begin{aligned}
		& t \tanh (\pi t) \nd \shskip t , \  & & \text{ if }  F_{\infty} \text{ is real}, \\
		&   t^2 \nd \shskip t ,   & & \text{ if }  F_{\infty} \text{ is complex}.
	\end{aligned}\right.
\end{equation} 
Moreover, we define 
\begin{equation}\label{1eq: defn of Pl(t)}
	\mathrm{Pl}  (t) = \left\{ \begin{aligned}
		& 4 \cosh (\pi t)   , & & \text{ if }  F_{\infty} \text{ is real}, \\
		& 8\pi \sinh (2\pi t) /   t   , \  & & \text{ if }  F_{\infty} \text{ is complex}.
	\end{aligned}\right.
\end{equation} 
Compared with \cite[(3.2)]{Qi-GL(3)}, we   have normalized $\mathrm{Pl}  (t)$ here by the factors $4$ and $8 \pi$. 
Let  $W_{ i t} $ be the spherical ($K_{\infty}$-invariant) Whittaker vector so that
\begin{equation}
	W_{i t} \begin{pmatrix}
		x  & \\ & 1
	\end{pmatrix} = \left\{ \begin{aligned}
		&   \|x  \|_{\infty}^{\frac 1 2}    K_{i t} (2\pi |x  |) , & & \text{ if }   F_{\infty} \text{ is real}, \\
		& \|x  \|_{\infty}^{\frac 1 2} K_{2 i t} (4\pi |x |)  , \  & & \text{ if }   F_{\infty} \text{ is complex}.
	\end{aligned}\right.
\end{equation}

\vskip 5 pt

\subsection{Hecke--Maass Cusp Forms}\label{sec: automorphic forms} 
%Next, we consider the space of   automorphic forms on $\PGL_2 (\frO) \backslash \PGL_2 (F_{\infty}) / K_{\infty}$. 

%\subsubsection{Preliminaries on Cusp Forms}

Fix an orthonormal basis $ \SB $   for the   cuspidal subspace of $L^2 (\PGL_2 (\frO) \backslash \PGL_2 (F_{\infty}) / K_{\infty})$ that consists of eigenforms for the Hecke algebra as well as the Laplacian operator (Hecke--Maass cusp forms). Each $f \in \SB $ transforms under a certain representation $\pi (i t_f)$ of $\PGL_2 (F_{\infty})$, for some $ t_f \in Y $. In general, we have the Kim--Sarnak bound in \cite{Blomer-Brumley}:
\begin{align}\label{2eq: Kim-Sarnak}
	|\Im (t_{f })| \leqslant \frac 7 {64},
\end{align}
but it is known that $t_f$ is real for $F = \BQ$ or $\BQ (\sqrt{d_F})$ with $d_F =  -3,-4,-7,-8, - 11  $ (see \cite[\S 7.6]{EGM}). Accordingly, define $Y_{\mathrm{KS}} =  (-\infty, \infty) \cup i \hskip -1.5 pt \left[ - \frac 7 {64}, \frac 7 {64} \right]$.
The Fourier expansion of $f $ is of the form:  
\begin{align}\label{2eq: Fourier expansion}
	f (  g_{\infty}) = \sum_{ n \shskip \in \shskip \frO' \smallsetminus \{0\}} \frac {a_f (n \frD)} {\sqrt{ \RN (n \frD) }} W_{i t_f} (\begin{pmatrix}
		n & \\ & 1
	\end{pmatrix} g_{\infty}), \quad g_{\infty} \in \GL_2 (F_{\infty}). 
\end{align}
As indicated by the notation, the Fourier coefficient $a_f (\frn) = a_f (n \frD)$  only depends on the ideal $\frn = n \frD$.  
Let  $\lambda_f (\frn)$ denote the $\frn$-th Hecke eigenvalue of $f$. It is known that $\lambda_f (\frn)$ are real.  We have the Hecke relation:
\begin{align}\label{3eq: Hecke relation}
	\lambda_f (\frn_1) \lambda_f (\frn_2) = \sum_{ \frd | (\frn_1, \frn_2) } \lambda_f \big(\frn_1 \frn_2 / \frd^2 \big) . 
\end{align}
 As usual, there is a constant $C_f$ so that 
\begin{align}\label{2eq: af = lambda f}
	a_f (\frn) = C_f \lambda_f (\frn)
\end{align}
for any nonzero integral ideal $\frn$. By the Rankin--Selberg method, we have
 \begin{align}\label{3eq: Cf2 = L(1, Sym2)}
		  {  |C_f|^2  } {  }  = \frac {   \mathrm{Pl} (t_f) } {  2  L(1, \mathrm{Sym^2} f) }  .
\end{align}   
%Note that we have included the factors $4$ and $8\pi$ in \eqref{1eq: defn of Pl(t)}.

\vskip 5 pt

\subsection{Kuznetsov Trace Formula}

%First, we    define the Kloosterman sum.  

%We have the Weil bound for Kloosterman sums:
%\begin{align}\label{3eq: Weil}
%	S (m, n ; c ) \Lt \sqrt{\RN (m \frD, n \frD, c  )} \sqrt{|\RN (c )|} \tau (c),
%\end{align}
%where the  brackets $(\cdot\, , \cdot\, , \cdot)$ denote greatest common divisor (of ideals).

\begin{defn}[Space of test functions] \label{defn: test functions}
	Let $S > \frac 1 2$. We set
	$ \mathscr{H} (S ) $ to be the space of functions $h (t) $ which extends to an even holomorphic function on the strip
	$\big\{ t + i \sigma    : | \sigma | \leqslant S \big\}$ such that  
	\begin{align*}
		h  (t + i \sigma) \Lt e^{-\pi |t|} (1+|t|)^{- N}, 
	\end{align*}
	holds uniformly for some $N > 6$. 
	
	%We also let $h : \BR \ra \BC $ denote the restriction of $h$ on the diagonal embedding of $ \BR$ in $\bfra$, namely, $ h (t) = \prod_{v | \infty} h_{v} (t)  $  for $ t \in \BR$. 
\end{defn}

\begin{defn}
	[Bessel kernel] \label{defn: Bessel kernel}
	
	Let  $s  \in  \BC  $. 
	
	{\rm(1)} When  $F_{\infty} = \BR$, for $x  \in \BR_+$ we define
	\begin{align*}
		&B_{s} (x)  = \frac {\pi} {\sin (\pi s) } \big( J_{-2 s} (4 \pi \sqrt {x }) - J_{2 s} (4 \pi \sqrt {x }) \big), \\
		&B_{s} (-x )   % = \frac {\pi} {\sin (\pi s) } \big( I_{-2 s} (4 \pi \sqrt {x }) - I_{2 s} (4 \pi \sqrt {x }) \big) 
		=   {4 \cos (\pi s)}    K_{2 s} (4 \pi \sqrt {x }) .
	\end{align*}
	
	{\rm(2)} When  $F_{\infty} = \BC$, for $z \in \BC^{\times}$ we define
	\begin{equation*}
		B_{s} (z ) =  \frac {2\pi^2} {\sin (2\pi s) } \big( { \textstyle  J_{-2 s} (4 \pi \sqrt {z}) J_{- 2s} (4 \pi \sqrt { \widebar z}) - J_{2 s} (4 \pi \sqrt {z}) J_{ 2s} (4 \pi \sqrt { \widebar z}) } \big). 
	\end{equation*}  
	
%	It is understood that when $  s \in \BZ $ or $  2 s \in \BZ $ in {\rm(1)} or {\rm(2)}, respectively,  the formulae above should be replaced by their limit. 
	 
\end{defn}	

The Kuznetsov trace formula of Bruggeman and Miatello in the spherical case is as follows. See \cite[Proposition 3.5]{Qi-GL(3)} or \cite[Proposition 1]{Venkatesh-BeyondEndoscopy}.

\begin{prop}[Kuznetsov trace formula]\label{prop: Kuznetsov}
	Let $h (t)$ be a test function in	$ \mathscr{H} (S) $   and define 
	\begin{align}\label{1eq: defn Bessel integral}
	  	\SDH = \int_{-\infty}^{\infty} h (t) \nd \shskip \mu (t), \quad \SDH (x) = \int_{-\infty}^{\infty} h (t)  B_{i t} (x ) \nd \shskip \mu (t), \quad \quad \text{$x\in F_{\infty}^{\times}$}. 
	\end{align} 
For nonzero integral ideals $\frm = m \frD$ and $  \frn = n \frD$ we have
	\begin{equation}\label{1eq: Kuznetsov} 
		\begin{aligned}
			 	\sum_{f \in  \SB  } \hskip -1pt \omega_f  h  ( t_f )   \lambda_f  ( \frm ) &    {\lambda_f  ( \frn )}   + \frac {1}  {4\pi} c_0 \int_{-\infty}^{\infty} \hskip -2pt  \omega (t) h ( t ) 
			\tau_{it} (\frm )  {\tau_{it} (\frn )} \shskip   \nd t  \\
		&	=       c_1 \delta_{\frm,   \frn}  \SDH  + c_2   \sum_{ \epsilon \, \in \shskip \frOO^{\sstimes} \hskip -1pt / \frOO^{\sstimes 2} } \sum_{c \, \in \shskip \frO \smallsetminus \{0\} } \frac {S ( m ; \epsilon n  ; c  ) } { |\RN (c  )| } \SDH \bigg( \frac {\epsilon m n } {  c^2    }  \bigg),
		\end{aligned}
	\end{equation}
	where 
	\begin{align}\label{1eq: omegas}
		\omega_f = \frac {|C_f|^2} {\mathrm{Pl}(t_f)} = \frac {  1  } { 2 L(1, \mathrm{Sym^2} f) }, \quad \quad \omega (t) =  \frac { 1 } {|\zeta_F(1+2it)|^2}, 
	\end{align} 
%the divisor function	$\tau_s (\frn)$ is defined in {\rm\eqref{1eq: defn of tau (n)}},
 $\delta_{\frm,   \frn} $ is the Kronecker $\delta$ that detects $\frm = \frn$, %the Kloosterman sum $S ( m ;   n  ; c  )$ is defined in {\rm\ref{2eq: defn Kloosterman KS}},  
 and   $c_0$, $c_1$, and $c_2$ are given by
	\begin{align} \label{3eq: constants, Q}
	c_0 = 1, \qquad c_1 = \frac {1} {8 \pi^{2 }  }  , \qquad c_2 =   \frac {1} {16 \pi^{2 } }  ,
	\end{align}
if $F = \BQ$, and 
\begin{align} \label{3eq: constants, C}
	c_0 =   \frac  { 2\pi  }  {w_F \sqrt{|d_F|} } , \quad c_1 = \frac {\sqrt{|d_F|}} {8 \pi^{3}  }  , \quad c_2 =   \frac {1} {16 \pi^{3}   }  ,
\end{align}
if $F = \BQ(\sqrt{d_F})$. 
\end{prop}

For our normalized $\mathrm{Pl}  (t)$ the constants in \cite[(3.18), (3.19)]{Qi-GL(3)} have been modified here accordingly.  Note that $c_0 = \gamma_1$ and $c_2 =   c_1 /2 \sqrt{|d_F|}$.

By the discussions below    \cite[Lemma 2.2]{Qi-Liu-LLZ}, it is known that the lower bound $|\zeta_F(1+2it)| \Gt  \log  (3+|t|) $ holds, and hence
\begin{align}\label{3eq: bound for omega(t)} 
	\omega (t) \Lt  \log^2 (3+|t|) . 
\end{align}

\subsection{Vorono\"i Summation Formula}

The Vorono\"i summation formula for the  divisor function $\tau (\frn) = \tau_0 (\frn)$ is as follows.  Compare \cite[(4.49)]{IK}.

\begin{prop}[Vorono\"i summation formula] \label{prop: Voronoi tau}	Let $ a, \widebar{a},  c \in \frO$ with $c \neq 0$ be such that $(a, c) = (1)$ and $a \widebar{a} \equiv 1 (\mod c)$. For $\varww (x) \in C^{\infty}_c (F^{\times}_{\infty}) $  we define its Hankel transform $ \widetilde{\varww}_{0} (y) $ with  Bessel kernel $B_0 $  {\rm(}as in Definition {\rm\ref{defn: Bessel kernel}}{\rm)}{\rm:}
	\begin{align}\label{3eq: Hankel, global}
		\widetilde{\varww}_{0} (y) =    \int_{F^{\sstimes}_{\scalebox{0.55}{$\infty$} } }  \varww (x) B_{0}    ( x y)   \nd x, \qquad  y \in F^{\times}_{\infty} ,
	\end{align} 
 and   define its associated Mellin integrals{\rm:} %$ \widetilde{\varww}_0 (0)$ and  $ \widetilde{\varww}_0' (0)$ by
\begin{align}\label{3eq: Mellin}
	\widetilde{\varww}_0 (0) =	 \int_{F^{\sstimes}_{\scalebox{0.55}{$\infty$} } }  \varww (x)   \nd x , \qquad \widetilde{\varww}_0' (0) =	 \int_{F^{\sstimes}_{\scalebox{0.55}{$\infty$} } }  \varww (x)    \log   \|x\|_{\infty}  \nd x . 
\end{align}   Then we have the identity
	\begin{align}\label{app: Voronoi, tau} 
		\begin{aligned}
			 \frac {{|\RN(c)|}}   {\sqrt{|d_F|}} \hskip -2 pt 	 \sum_{n \shskip \in \shskip \frO '}     \hskip -2 pt  \psi_{\infty} \Big( \frac {a n} {c} \Big)        \tau  (n  \frD)   \varww   (n )        
	 	=	    \gamma_{1}  \widetilde{\varww}_0' (0)  + 2 \bigg(   \gamma_{0}  -\gamma_{1}  \log  \frac {{|\RN(c)|}}   {\sqrt{|d_F|}}     \bigg) \widetilde{\varww}_0 (0) & \\
			   	+   \frac 1 {\sqrt{|d_F|}}   \sum_{ n \shskip \in  \shskip \frO' \smallsetminus \{0\} }  \hskip -2 pt  \psi_{\infty} \Big( \hskip -2 pt - \frac {\widebar{a} n} {c} \Big) \tau (n \frD ) \widetilde{\varww}_{0} \Big(\frac {n}{c^2}\Big) & ,
		\end{aligned}
	\end{align}
	where the constants $\gamma_{0} $ and  $\gamma_{1} $ are defined as in {\rm\eqref{2eq: zeta (s), s=1}}. 
\end{prop}

\begin{proof}
	Apply \cite[Corollary 1.4]{Qi-VO} with $\zeta = a/c$, $\fra = (1)$, $S = \big\{ \frp : \frp | (c) \big\}$, and  $\frb = (c^2)$. Note that every $\zeta \in F$ may be expressed as a fraction $\zeta = a/c$ with $(a, c) = (1)$ since the class number $h_F = 1$. 
\end{proof}

It will be more convenient to interpret the zero frequency as the limit:
\begin{align}\label{3eq: limit for 0}
	\gamma_{1}  \widetilde{\varww}_0' (0)  + 2 \bigg(   \gamma_{0}  -\gamma_{1}  \log  \frac {{|\RN(c)|}}   {\sqrt{|d_F|}}     \bigg) \widetilde{\varww}_0 (0) = \lim_{s \ra 0} \sum_{\pm} \frac {\zeta_F (1\pm 2s)  \widetilde{\varww}_{\pm s}  (0)} {  |   \RN ( c )^2 / d_F  | ^{ \pm  s}  }, 
\end{align}
where  $ \widetilde{\varww}_{s} (0) $ is the Mellin transform 
\begin{align}  \widetilde{\varww}_{s} (0) = \int_{F^{\sstimes}_{\scalebox{0.55}{$\infty$} } }  \varww (x) \|x\|_{\infty}^{s}    \nd x .
\end{align}
See \cite[Theorem 1.3]{Qi-VO}.

%For sums involving the divisor function, one may find the following estimates useful as $\tau (\frn) \Lt_{\vepsilon} \RN (\frn)^{\vepsilon}$. 

%\section{Preliminaries on the Gamma Function and Bessel Functions}

\section{Approximate Functional Equations}

Let $f \in \SB $ be a {(spherical)} Hecke--Maass cusp form for $ \PGL_2 (\frO) $ with Hecke eigenvalues $\lambda_f (\frn)$ and Archimedean parameter  $ t_f \in  Y$. 
The $L$-function attached to $ f    $ is defined by
\begin{equation}
	L (s,  f   ) = \sum_{\frn \shskip \subset \shskip \frO }  \frac {\lambda_f(  \frn)  } {\RN(\frn)^{  s} }.
\end{equation}
The completed $L$-function for $ f $ is
$\Lambda (s,  f   ) =  \RN(\frD)^{ s } \gamma (s, t_f  ) L (s,  f  ) $, where 
\begin{equation}\label{4eq: defn of gamma (s, f)}
	\gamma  (s, t) =  (N \pi )^{-  N s } \Gamma \bigg(\frac {N(s- i t)} 2 \bigg)   \Gamma \bigg(\frac {N(s+ i t)} 2 \bigg) .
\end{equation}
Recall that $N = 1$ or $2$ according as $F$ is rational or imaginary quadratic. It is known that $\Lambda (s,  f )$ is entire and  has the functional equation
\begin{align*}
	\Lambda (s,  f  ) = \Lambda   (1-s,  f   ).
\end{align*} 
For $\mathrm{Re} (s) > 1$, it follows from the Hecke relation \eqref{3eq: Hecke relation} that
\begin{align*}
	\begin{aligned}
		L (s,  f   )^2 & = \mathop{\sum\sum}_{\frn_1, \frn_2 \subset \frO}    \sum_{   \frd | (\frn_1, \frn_2)  } \frac { \lambda_f \big(\frn_1 \frn_2 / \frd^2 \big) } {{\RN(\frn_1 \frn_2)^s }}   = \sum_{ \frd \subset \frO} \frac 1 {\RN (\frd)^{2s} } \mathop{\sum\sum}_{\frn_1, \frn_2 \subset \frO} \frac {\lambda_f (\frn_1 \frn_2)   } {\RN (\frn_1 \frn_2)^s} , 
	\end{aligned}
\end{align*} 
and hence 
\begin{align}\label{5eq: L(s,f) square}
	 L (s,  f   )^2 = \zeta_F (2s) \sum_{\frn \subset \frO} \frac {\lambda_f (\frn) \tau (\frn) } {\RN (\frn)^s}, 
\end{align}
where $\tau (\frn)$ is the divisor function. 

Similarly, if $\lambda_f (\frn)$ are replaced by $\tau_{i t} (\frn) $, then  
\begin{align}\label{5eq: L (E)}
	  \zeta_F   (s+it )   \zeta_F   (s-it ) = \sum_{\frn \shskip \subset \shskip \frO }  \frac {\tau_{it} (  \frn)  } {\RN(\frn)^{  s} } ,
\end{align}
and  
\begin{align}\label{5eq: L (E) square}
	\zeta_F   (s+it )^2   \zeta_F   (s-it )^2 = \zeta_F (2s) \sum_{\frn \shskip \subset \shskip \frO }  \frac {\tau_{it} (  \frn)  \tau (\frn)} {\RN(\frn)^{  s} } .
\end{align}

We have the Approximate Functional Equations for   $L (s,    f )$ and $L (s,    f )^2$ (see \cite[Theorem 5.3]{IK}):
\begin{equation}
	\label{5eq: AFE, 1} 
	L \big(\tfrac 1 2,    f \big)   =  2 \sum_{\frn \shskip \subset \shskip \frO }  \frac {   \lambda_f  (\frn )  } {\sqrt{\RN  (  \frn   )}  }     V_1  \big(  \RN  \big(  \frn  \frD^{-1} \big); t_f  \big)   , 
\end{equation}
\begin{equation}
	\label{5eq: AFE, 2} 
	L \big(\tfrac 1 2,    f \big)^2   =  2 \sum_{\frn \shskip \subset \shskip \frO }  \frac {   \lambda_f  (\frn ) \tau (\frn) } { \sqrt{\RN  (  \frn   )}  }     V_2  \big(  \RN  \big(  \frn  \frD^{-2} \big); t_f  \big)   , 
\end{equation}
with 
\begin{equation}\label{5eq: def of V1 (y, t)} 
		V_1 (y; t ) = \frac 1  {2 \pi i} \int_{(3)} 
		G  (v, t)   y^{ -  v} \frac { \nd   v } {v} ,   
\end{equation}
\begin{equation}\label{5eq: def of V2 (y, t)}
	V_2 (y; t ) = \frac 1  {2 \pi i} \int_{(3)} 
	G  (v, t)^2 \zeta_F (1+2v)   y^{ - v} \frac { \nd v } {v} ,   
\end{equation}
for $y > 0$, where 
\begin{align}\label{4eq: def G}
	G (v, t) = \frac {\gamma \big(\frac 1 2 + v , t \big)  }  {\gamma \big(\frac 1 2, t   \big)   } \cdot  e^{v^2  }  .
\end{align}   
In parallel, we have
\begin{align}
	\label{5eq: AFE zeta, 1} 
	\left| \zeta_F \big(\tfrac 1 2 + it \big) \right|^2  &  =  2 \sum_{\frn \shskip \subset \shskip \frO }  \frac {   \tau_{it}  (\frn )  } {\sqrt{\RN  (  \frn   )}  }     V_1  \big(  \RN  \big(  \frn  \frD^{-1} \big); t  \big) + O  \big( e^{-t^2/2} \big)  , \\
	\label{5eq: AFE zeta, 2} 
	\left| \zeta_F \big(\tfrac 1 2 + it \big) \right|^4  & =  2 \sum_{\frn \shskip \subset \shskip \frO }  \frac {   \tau_{it}  (\frn ) \tau (\frn) } { \sqrt{\RN  (  \frn   )} }     V_2  \big(  \RN  \big(  \frn  \frD^{-2} \big); t  \big)  + O  \big( e^{-t^2 } \big) , 
\end{align} 
in which the errors arise from the polar terms.

\begin{lem}\label{lem: afq}  
 For $t \in Y_{\mathrm{KS}}$  define 
	\begin{align*}%\label{1eq: defn conductor}
		\RC  (t) =        \sqrt{\tfrac 1 4 +  t^2}   .
	\end{align*}

{\rm(1)} Let  $ U > 1  $,  $A > 0$, and $\vepsilon > 0$. We have
	\begin{align}\label{1eq: derivatives for V(y, t), 1}
		%p \left(\tfrac 1 2, t \right)^2 
		V_1 (y; t ) \Lt_{ A  } %\RP (t)^{12 } 
		\bigg(  1 + \frac {y} {\RC (t)^{N} } \bigg)^{-A} , \quad V_2 (y; t ) \Lt_{ A  } %\RP (t)^{12 } 
		\bigg(  1 + \frac {y} {\RC (t)^{2N} } \bigg)^{-A},
	\end{align}   and
	\begin{align}
		\label{1eq: approx of V1}
		V_1  (y; t) & = \frac 1 {2   \pi i } \int_{ \vepsilon - i U}^{\vepsilon + i U}  G (v, t)    y^{ - v}   \frac {\nd v} {v} + O_{\vepsilon } \bigg( \frac {\RC (t)^{  \vepsilon} } {y^{ \vepsilon} e^{U^2 / 2} } \bigg),  \\
		\label{1eq: approx of V2}
		V_2  (y; t) & = \frac 1 {2   \pi i } \int_{ \vepsilon - i U}^{\vepsilon + i U}  G (v, t)^2 \zeta_F (1+2v)    y^{ - v}   \frac {\nd v} {v} + O_{\vepsilon } \bigg( \frac {\RC (t)^{  \vepsilon} } {y^{ \vepsilon} e^{U^2  } } \bigg) .
	\end{align}  

{\rm(2)} Define 
\begin{align}\label{4eq: defn of psi}
	\psi  (t) = \frac {N} 2 \lp \frac {\Gamma'} {\Gamma} \lp \frac {N (1+2 i t)} {4} \rp +   \frac {\Gamma'} {\Gamma} \lp \frac {N (1 - 2 i t)} {4} \rp -  2   \log (N \pi) \rp . 
\end{align}
{\rm(}This $\psi  $   should not be confused with the additive character as it stands here for the digamma function.{\rm)}   
We have
\begin{align}\label{4eq: asymptotic for V1}
	V_1 (y; t) = 1 + O_{A }   \lp \lp \frac {y}  {\RC  (t)^N} \rp \hskip -8.5 pt {\phantom{\Big)}}^{A } \rp  , 
\end{align}
for $1 \Lt y <  \RC (t)^{N}$, and
\begin{align}\label{4eq: asymptotic for V2}
	V_2 (y; t) = \gamma_{0} + \gamma_{1} \lp \psi (t)  -  \log   \sqrt{y}  \rp     +  O_{A }   \lp \lp \frac {y}  {\RC  (t)^{2N}} \rp \hskip -8.5 pt {\phantom{\Big)}}^{A } \rp, 
\end{align}
for $1 \Lt y <  \RC (t)^{2N}$, where   $\gamma_{0}$ and  $\gamma_{1}$ are defined as in {\rm\eqref{2eq: zeta (s), s=1}},   and 
	  \begin{align}\label{4eq: asymp of psi}
	  	\psi (t) = N \log \lp \RC (t) / 2\pi \rp  + O \big(1/ \RC (t)^2 \big).
	  \end{align} 
\end{lem}

%Finally, the following asymptotic formulae will be useful for determining the main terms. 

\begin{proof}
	The  asymptotics in (1) are analogous to those in \cite[Lemma 5.1 (1)]{Qi-GL(3)}. See also \cite[Proposition 5.4]{IK}, \cite[Lemma 1]{Blomer}, and \cite[Lemma 3.7]{Qi-Gauss}. 
	To derive  \eqref{4eq: asymptotic for V1} and \eqref{4eq: asymptotic for V2}, we choose $U = \sqrt{\RC (t)}  $, say, and shift the integral contour in \eqref{1eq: approx of V1} and \eqref{1eq: approx of V2} from $\mathrm{Re}(v) = \vepsilon$  further down to $\mathrm{Re}(v) = - A $; the main term    is the residue from the pole at $v = 0$ while the error term is from the Stirling formula. Note that the integrand in \eqref{5eq: def of V2 (y, t)} has a double pole at $v = 0$, and its residue may be computed using
	\begin{align*}%\label{2eq: zeta (s), s=1}
		\zeta_F (1 + 2v) = \frac {\gamma_{1}} {2 v} + \gamma_{0} + O (|v|), \qquad v \ra 0, 
	\end{align*}  
by \eqref{2eq: zeta (s), s=1}, and
	\begin{align*}%\label{2eq: zeta (s), s=1}
		G (v, t) = 1 +     \psi  (t)  v + O \big(|v|^2 \big), \qquad  v \ra 0,
	\end{align*} 
by \eqref{4eq: defn of gamma (s, f)} and \eqref{4eq: def G}. 
Moreover, \eqref{4eq: asymp of psi} follows readily from 
\begin{align*}%\label{1eq: asymp for psi}
	\frac {\Gamma'} {\Gamma} (s) = \log s - \frac 1 {2s} + O \bigg(\frac 1 {|s|^2} \bigg),  
\end{align*}
for $|s| \ra \infty$ and $|\arg (s)| \leqslant \pi -\delta < \pi$. 
\end{proof}

\section{Choice of Weight Function}\label{sec: choice of h}

%Now we recollect the  definition of the weight function $k  (t)$ as in \eqref{1eq: defn of kq}.

\begin{defn}\label{def: weight k}
	Let  $1 \Lt T^{\vepsilon} \leqslant M \leqslant T^{1-\vepsilon} $. Define the function  
	\begin{align}\label{5eq: defn k(nu)}
		k  (t) = k_{T, M} (t) =   e^{- (t  - T)^2 / M^2} + e^{-(t + T)^2 / M^2}   .
	\end{align}   
\end{defn}

%Now we define the first and second smoothly weighted twisted moments:
%\begin{align}
% \SM_q (\frm ) =		\SM_q (\frm; T, M) = \sumh_{f \in  \SB  } \hskip -1pt   k  ( t_f ) \lambda_f  ( \frm )    L \big( \tfrac 1 2 , f \big)^q .  
%\end{align} 

Next we introduce an unsmoothing process as in \cite[\S 3]{Iviv-Jutila-Moments} by an average of the weight $k_{T, M} $ in the $T$-parameter.  

\begin{defn}
	Let $ 3 H \leqslant T  $ and $T^{\vepsilon} \leqslant M \leqslant H^{1-\vepsilon}  $. Define  
	\begin{align}\label{5eq: defn w(nu)}
		\varww (t) =  \varww_{T, M, H} (t) = \frac 1 { \sqrt{\pi}  M} \int_{\, T- H}^{T + H} k_{K, M} (t) \nd K  .
	\end{align}
\end{defn}

Let  $\chi (t) = \chi^{}_{T, H} (t)$ denote the characteristic function for $ ||t| - T | \leqslant H$ ($t$ real).  
By adapting the arguments in \cite[\S 3]{Iviv-Jutila-Moments} (see also \cite[\S 3]{BHS-Maass}\footnote{Note that the $1$ in \cite[(3.4)]{BHS-Maass} should be the characteristic function.}), it is easy to prove that
$\varww (t) - 1 $ is exponentially small
if $ | |t| - T | \leqslant H - M^{1+\vepsilon} $, 
\begin{align*}%\label{5eq: weight, 2}
	\varww (t) - \chi (t) = O \lp   \frac {M^3} {\big(  M  + \min  \big\{ ||t|-T \pm H|  \big\} \big)^3 } \rp,
\end{align*}
if $ ||t| - T \pm H | \leqslant M^{1+\vepsilon} $, % (this holds even for  $H$ close to $T$ and for $t \in i \big( \hskip -1.5pt -\frac 1 2 , \frac 1 2 \big)$), 
and $\varww (t)$ is exponentially small if otherwise. From these, along with \eqref{2eq: subconvex} and \eqref{3eq: bound for omega(t)}, one  may prove  the following lemmas (compare \cite[(3.6), (3.7)]{Iviv-Jutila-Moments}). 
 
\begin{lem}\label{lem: unsmooth, 1}
Let  $\lambda $  be a real constant.   	Suppose that $a_f \geqslant 0$ and that 
	\begin{align*}
		 \sumh_{f \in  \SB}  k   (t_f) a_f = O_{\lambda, \vepsilon}  \big( M T^{\lambda}  \big) 
	\end{align*}
for any $M$ with $ T^{\vepsilon} \leqslant M \leqslant T^{1-\vepsilon}  $. Then for  $ M^{1+\vepsilon} \leqslant 3 H \leqslant T  $ we have
\begin{align*}
	\mathop{\sumh_{f \in  \SB} }_{  |t_f - T| \shskip \leqslant H}    a_f =  \sumh_{f \in  \SB}  \varww  (t_f) a_f + O_{\lambda, \vepsilon} \big( M T^{\lambda} \big) . 
\end{align*}
\end{lem}

Note that we obtain \cite[Lemma 3.1]{BHS-Maass} by applying Lemma \ref{lem: unsmooth, 1}  with  $M = H^{1-\vepsilon}$ and $a_f = \delta_{L (\frac 1 2 , f) \neq 0}$ (the Kronecker $\delta$ that detects $L (\frac 1 2 , f) \neq 0 $). 

\begin{lem}\label{lem: unsmooth, 2}
	For  $ M^{1+\vepsilon} \leqslant 3 H \leqslant T  $ we have
	\begin{align*}
	2 \int_{\, T-H}^{T+H}        \frac {\left| \zeta_F \big(\tfrac 1 2 + it \big) \right|^{2q} } { | \zeta_F  (1 + 2 it  )  |^2 } \nd t  = 	  \int_{-\infty}^{\infty}    \frac {\left| \zeta_F \big(\tfrac 1 2 + it \big) \right|^{2q} } { | \zeta_F  (1 + 2 it  )  |^2 }
		\varww (t) \shskip   \nd t + O_{  \vepsilon} \big( M T^{2q N \theta + \vepsilon} \big), 
	\end{align*}
where $\theta$ is a  sub-convex exponent for $\zeta_F (s)$ as in {\rm\eqref{2eq: subconvex}}. 
\end{lem}

 { \large \part{Analysis of Integrals} \label{part: analysis}}

In the subsequent sections, we shall analyze the Bessel integrals, their Hankel and   Mellin integral transforms over  $F_{\infty} = \BR$ or $\BC$.    Henceforth, $x$, $y$ will always stand for real variables, while $z$, $u$ for complex variables.

\vskip 5pt

%\section{Analysis of   Plancherel Integrals}

%We start with the Plancherel integral $\SDH$ defined in \eqref{1eq: defn Bessel integral}.  

\section{Asymptotics for Bessel Kernels}

Let $ B_{s}(x) $  and $B_{s} (z)$  be the real and complex Bessel kernels  as in Definition \ref{defn: Bessel kernel}, respectively.

By the works in \cite{Qi-Liu-LLZ,Qi-GL(3)}, the   Bessel integrals $\SDH (x)$ and $\SDH (z)$ in the Kuznetsov trace formula are well understood, and  their results will be recollected in the next section.  In this section, we are mainly concerned with the Bessel kernel   $B_0 (x)$ or  $B_0 (z)$ for the Hankel transform arising in the Vorono\"i summation formula.  

In view of Definition \ref{defn: Bessel kernel}, the connection formulae in \cite[3.61 (1), (2)]{Watson} may be applied to deduce  
\begin{align}\label{4eq: formula B0, R, 2}
	B_0 (x) =  \pi i \big( H_0^{(1)} (4\pi \sqrt{x}) - H_0^{(2)} (4\pi \sqrt{x}) \big), \quad B_0 (-x) = 4 K_0 (4\pi \sqrt{x}), 
\end{align}
and
\begin{align}\label{4eq: formula B0, C, 2}
	B_0 (z) =  \pi^2 i \big( { \textstyle  H_{0}^{(1)} (4 \pi \sqrt {z}) H_{0}^{(1)} (4 \pi \sqrt { \widebar z}) - H_{0}^{(2)} (4 \pi \sqrt {z}) H_{0}^{(2)} (4 \pi \sqrt { \widebar z}) } \big) .
\end{align} 
By the asymptotic expansions in \cite[7.2 (1, 2), 7.23 (1)]{Watson}, for any non-negative integer $K$, there are smooth functions $ W_0 (x) $ and $W_0 (z)$ (depending on $K$) with 
\begin{align}\label{5eq: bounds for W0}
	x^j \frac {\nd^j W_0 (x)} {\nd x^j} \Lt_{j, K} 1, \qquad   z^j \widebar{z}^k \frac {\partial^{j+k} W_0 (z)} {\partial z^j \partial \widebar z^k } \Lt_{j, k, K} 1 ,  
\end{align}
 such that 
\begin{align} 
	\label{4eq: asymptotic, R+}		& B_0 (x) =   \sum_{ \pm} \frac {e (\pm (2 \sqrt{x} + 1/8))} {   \sqrt[4]{x \phantom{|\hskip -2 pt}} } W_0 (\pm \sqrt{  x}) + O_{  K} \bigg(  \frac 1 {x^{(2K+1)/4}} \bigg),  \\
	\label{4eq: asymptotic, R-}		& B_0 (-x) = O  \bigg( \frac {\exp (-4\pi \sqrt{x})} {\sqrt[4]{x \phantom{|\hskip -2 pt}}} \bigg),  
\end{align}
for $x > 1$, and 
\begin{align}\label{4eq: asymptotic, C}	
	& B_0 (z) =   \sum_{ \pm} \frac {e (\pm 4 \shskip \mathrm{Re} \sqrt{z})} {\sqrt{|z|} } W_0 (\pm \sqrt{ z}) + O_{ K} \bigg(  \frac 1 {|z|^{(K+1)/2}} \bigg),   
\end{align}
for $|z| > 1$. 

\begin{rem}
	For the real case, \eqref{4eq: asymptotic, R+} has a cleaner form without the error term. For the complex case, however, the error term must be included in {\rm\eqref{4eq: asymptotic, C}}, for the two product functions in {\rm\eqref{4eq: formula B0, C, 2}} are {\it not} individually well defined on $\BC \smallsetminus \{0\}$. 
\end{rem}

\section{Properties of Bessel Integrals}

For $1 \Lt T^{\vepsilon} \leqslant M  \leqslant T^{1-\vepsilon} $, let $k (t) = k_{T, \shskip M} (t)$ 
be the weight function as defined in \S \ref{sec: choice of h}.  Define 
\begin{align}\label{5eq: defn h(nu)}
	h^q (t; v) = h^q_{T, \shskip M}  (t; v) =  
	 k_{T, \shskip M}  (t) G (v, t)^q ,
\end{align} 
with $\mathrm{Re}(v) = \vepsilon$ and $|\mathrm{Im}(v)| \leqslant \log T$. Note that $ h^q (t; v) $ lies in the space $\mathscr{H} \big(\frac 1 2 + \vepsilon \big)$ as in Definition \ref{defn: test functions}. Since $q$ and $v$ are inessential to our analysis, we shall simply write  $h (t) = h^q (t; v)$. 
Let $\SDH (x) $ or $\SDH (z)$ be its associated Bessel integral (see \eqref{1eq: defn Bessel integral}) defined by
\begin{align}\label{7eq: defn of H(x)}
	\SDH (x) = \int_{-\infty}^{\infty} h (t) B_{i t} (x) t  \tanh (\pi t ) \nd \shskip t, \quad \SDH (z) = \int_{-\infty}^{\infty} h (t) B_{i t} (z) t^2  \nd \shskip t . 
\end{align}   
The following results for  $\SDH (x) $ and $\SDH (z)$ are essentially  established in \cite{Qi-Liu-LLZ} and \cite[\S 8]{Qi-GL(3)} in   different settings (see also \cite{Iviv-Jutila-Moments,XLi2011,Young-Cubic} for the real case). For the complex case, however, it will be more convenient to work here in the Cartesian coordinates. 

%First, we have crude estimates for $\SDH (x)$ and $\SDH (z)$ as below. 

The estimates above may be derived from shifting the integral contour to  $\Im (t) = \frac 1 2 +   \vepsilon$. See \cite{Young-Cubic} and \cite{Qi-Liu-LLZ}. 

\begin{lem}\label{lem: H(x), |z|>1}
	
	There exists a Schwartz function $ g (r)$ satisfying $g^{(j)} (r) \Lt_{j, \shskip A,  \shskip \vepsilon}  (1 + |r| )^{-A}$ for any $j, A \geqslant 0$, and such that
	
	{\rm(1)} if $F_{\infty} $ is real,  then $ \SDH (x) = \SDH_{  +}   (x) + \SDH_{  -}  (x) + O   (T^{-A} ) $ for $|x| >  1$, with 
	\begin{equation}\label{8eq: H+natural}
		\SDH_{  \pm}   (  x^2) =   MT^{1+\vepsilon}  
		\int_{- M^{\vepsilon} / M}^{M^{\vepsilon}/ M}   g (    {   M r} )  e( Tr / \pi \mp 2 x \cosh r  ) \nd r,
	\end{equation}
	and 
	\begin{equation}\label{8eq: H-natural}
		\SDH_{  \pm}   (- x^2) =   MT^{1+\vepsilon}  
		\int_{- M^{\vepsilon} / M}^{M^{\vepsilon}/ M}   g (    {   M r} )  e( Tr / \pi \pm 2 x \sinh r  ) \nd r,
	\end{equation}
	for $x > 1${\rm;}
	
	{\rm(2)} if $F_{\infty} $ is complex, then $ \SDH (z) = \SDH_{  +}   (z) + \SDH_{  -}  (z) + O   (T^{-A} ) $ for $|z| > 1$, with 
	\begin{align}\label{8eq: H-sharp(z)}
		\SDH_{\ssstyle \pm }    (  z^2 ) =  M T^{2+\vepsilon} \int_0^{ \pi}   \hskip -1 pt
		\int_{- M^{\vepsilon} / M}^{M^{\vepsilon}/ M}   g  ( M r )  
		e  (2 T r/ \pi \mp 4 \mathrm{Re} (z \trh ( r, \omega ) )     )    \nd r \shskip \nd \omega,
	\end{align} 
	for   $ \arg (z) \in [0, \pi)$, 
	where $\trh  (r, \omega )$ is the ``trigonometric-hyperbolic" function defined by 
	\begin{align}\label{8eq: trh function}
		\trh  (r, \omega ) =    \cosh r \cos \omega + i \sinh r \sin \omega  .
	\end{align} 
	
	Furthermore, 
	
	{\rm (3)} for real $x$ with $1 < |x| \Lt T^2 $, we have $ \SDH (x) = O  (T^{-A})${\rm;}
	
	{\rm (4)} for complex $z$ with $1 < |z| \Lt T^2 $, we have $ \SDH (z) = O  (T^{-A})${\rm;}
	
	{\rm(5)} for real $x$ with $|x| \leqslant 1$, we have 
	\begin{equation}\label{7eq: crude bound for H, R}
		\SDH (x) \Lt_{A , \shskip \vepsilon}   M T^{1 - 2 A } \sqrt{|x|}   ;
	\end{equation}

	{\rm(6)} for complex $z$ with $|z| \leqslant 1$,  we have  
	\begin{equation}\label{7eq: crude bound for H, C}
		\SDH (z) \Lt_{A , \shskip \vepsilon} M T^{2 - 4 A  } |z|   .
	\end{equation}  
\end{lem}

\begin{rem}
	In {\rm\cite{Qi-GL(3)}},  for the proof in the case $|x| \leqslant 1$ or $|z| \leqslant 1$ a certain polynomial  is introduced to annihilate the poles of the gamma factor, but it is redundant because  the residues of the integrand in {\rm\eqref{7eq: defn of H(x)}} at these poles are actually exponentially small  in view of $|\mathrm{Im}(v)| \leqslant \log T$. 
\end{rem}

In the real case,  Lemma \ref{lem: H(x), |z|>1} (3)  may be strengthened for $x > 1$ as follows. 

\begin{lem}\label{lem: x > MT}
	We have $ \SDH_{  \pm}  (x) = O \big(T^{-A}\big) $ for $1 < x \leqslant M^{2-\vepsilon} T^2    $. 
\end{lem}

%For the convenience of applying partial integration to the $r$-integral, one may modify $g (r) $ by a suitable partition of unity so that it is supported in $|r| \leqslant M^{\vepsilon}$. 

\section{Analysis of   Hankel Transforms}\label{sec: Hankel}

Let $\varww  (x) \in C_c^{\infty} [1, 2] $  satisfy $ \varww^{(j)} (x)  \Lt_{j} (\log T)^{j} $ for all $j \geqslant 0$. For  $|\varLambda| \Gt T^2$, define 
\begin{align}\label{11eq: defn of w (x, Lmabda), R}
	\varww (x, \varLambda ) = \varww (|x|) \SDH  ( \varLambda x  ) ,
\end{align}
if $F_{\infty}$ is real, and
\begin{align}\label{11eq: defn of w (z, Lmabda), C}
	\varww (z, \varLambda ) = \varww (|z|) \SDH  ( \varLambda  z ) ,
\end{align}
if $F_{\infty}$ is complex. Let $\widetilde {\varww}_0 (y , \varLambda )$ and $\widetilde {\varww}_0 ( u , \varLambda )$ be their Hankel transform  defined by
\begin{align}\label{9eq: Hankel}
	& \widetilde {\varww}_0 (y , \varLambda ) = \int {\varww} (x , \varLambda ) B_0 (xy) \nd x, \quad \widetilde {\varww}_0 (u , \varLambda ) = \viint {\varww} (z , \varLambda ) B_0 (z u) \nd z . 
\end{align}
First of all,  let us assume $\varLambda > 0$ with no loss of generality, as
\begin{align}\label{w(y, -L) = w (-y, L)}
\widetilde {\varww}_{0} (y ,  \varLambda )= \widetilde {\varww}_{0} ( \epsilon y , \epsilon \varLambda ), \qquad 	\widetilde {\varww}_{0} (u ,   \varLambda )= \widetilde {\varww}_{0} (  {\epsilon} u , \epsilon \varLambda ) , 
\end{align}  
for any   $\epsilon \in \Fx_{\infty}$ with $ |\epsilon| = 1 $.

\begin{lem}\label{lem: Hankel}
	Suppose that     $ \varLambda  \Gt T^2$. 
	
{\rm(1)} When $F_{\infty}$ is real, for  $ y \geqslant T^{\vepsilon}$ we have
	\begin{align}\label{10eq: tilde w = Phi, R}
		\widetilde {\varww}_0 ( \pm y , \varLambda) =   \frac{MT^{1+\vepsilon} } { \sqrt[4]{y \phantom{|\hskip -2 pt}} }  
		    \Psi^{\pm} \big( \sqrt{y /  \varLambda } , \sqrt{  \varLambda }\big)  
		+ O \big(T^{-A} \big) ,
	\end{align} 
 	with
	\begin{align}\label{10eq: Phi+ (x), R}
		\Psi^{+} (x, \varDelta) =   \int_{- M^{\vepsilon} / M}^{M^{\vepsilon}/ M} e( Tr / \pi )  g (    {   M r} )  \widehat{V}  ( \varDelta  (x - \cosh r )   )  \nd r ,
	\end{align}  
or $\Psi^{+} (x, \varDelta) = 0$ according as   $\varDelta  >  M^{1-\vepsilon} T$ or not, and 
\begin{align}\label{10eq: Phi- (x), R}
	\Psi^{-} (x, \varDelta) \hskip -1pt = \hskip -2pt  \int_{- M^{\vepsilon} / M}^{M^{\vepsilon}/ M} \hskip -1pt e( Tr / \pi ) g (    {   M r} )  \big( \widehat{V}  ( \varDelta  ( x \hskip -1pt + \hskip -1pt \sinh r   )   ) \hskip -1pt + \hskip -1pt \widehat{V}  ( \varDelta  ( x \hskip -1pt - \hskip -1pt  \sinh r   )   )   \big)  \nd r ,
\end{align}  
where $ \widehat{V} (x) $ is a  Schwartz function  satisfying  
\begin{align}\label{8eq: Schwartz, R}
	\frac{\nd^{j} \widehat{V}  (x) } {\nd x^{j}}   \Lt_{j, A} \lp 1 + \frac   {|x|}   {\log T}  \rp^{- A}   
\end{align}
 for any $j, A \geqslant 0$. 
 
 {\rm(2)} When $F_{\infty}$ is complex, for  $ |u| \geqslant T^{\vepsilon}$ we have
 \begin{align}\label{10eq: tilde w = Phi, C}
 	\widetilde {\varww}_0 ( u , \varLambda) =   \frac{MT^{2+\vepsilon} } {\sqrt{|u|} } 
 	\Psi  \big( \sqrt{ u /  \varLambda } , \sqrt{  \varLambda }  \big)  \nd \omega
 	+ O \big(T^{-A} \big) ,
 \end{align}  	with
 \begin{align}\label{10eq: Phi (x), C}
 	\Psi  (z, \varDelta) =  \int_0^{2 \pi}  \int_{- M^{\vepsilon} / M}^{M^{\vepsilon}/ M} e( 2 Tr / \pi )  g (    {   M r} )  \widehat{V}  ( \varDelta  ( z -  \trh (r, \omega)    )   )  \nd r \nd \omega ,
 \end{align}    
 where $ \widehat{V} (z) $ is a Schwartz function  satisfying  
 \begin{align}\label{8eq: Schwartz, C}
 	\frac{\partial^{j + k} \widehat{V}  (z) } {\partial z^{j} \partial \widebar{z}^{k} }  \Lt_{j, k, A} \lp 1 + \frac   {|z|}   {\log T }  \rp^{- A}  
 \end{align}
 for any $j, k,  A \geqslant 0$. 
\end{lem} 
 
\begin{proof}

First, let $F_{\infty}$ be real. %	We start with the case   $ y > T^{\vepsilon} $, so that  $B_0 (\pm xy)$ is either oscillatory or exponentially small. 
By \eqref{4eq: asymptotic, R+} (with $K$ large in terms of $\vepsilon$ and $A$), \eqref{4eq: asymptotic, R-}, and \eqref{8eq: H+natural}, \eqref{8eq: H-natural} in Lemma   \ref{lem: H(x), |z|>1} (1), along with the substitution $\pm  2 \sqrt{ x} \ra x$,  it follows that,  up to a negligible error,   $\widetilde {\varww}_0 (    y , \varLambda ) $ or $\widetilde {\varww}_0 (  -  y , \varLambda ) $ becomes the sum of 
	\begin{align*}
		\frac{MT^{1+\vepsilon} } { \sqrt[4]{y \phantom{|\hskip -2 pt}} }  
		\int_{- M^{\vepsilon} / M}^{M^{\vepsilon}/ M}  e( Tr / \pi ) g (    {   M r} )  \lp \int_{-\infty}^{\infty} V (x) e \big( \hskip -1pt -   x \big( \sqrt{y} \pm \sqrt{\varLambda} \cosh r   \big)  \big) \nd x \rp \nd r ,  
	\end{align*}
or
\begin{align*}
	\frac{MT^{1+\vepsilon} } { \sqrt[4]{y \phantom{|\hskip -2 pt}} }  
	\int_{- M^{\vepsilon} / M}^{M^{\vepsilon}/ M}  e( Tr / \pi ) g (    {   M r} )  \lp \int_{-\infty}^{\infty} V (x) e \big( \hskip -1pt -   x \big( \sqrt{y} \mp \sqrt{\varLambda} \sinh r   \big)  \big) \nd x \rp \nd r ,  
\end{align*}
respectively, where $V (x)$ is a certain smooth weight function  supported in $|x| \in [1/2, 1/ \sqrt{2}]$ with $$   V^{(j)} (x)  \Lt_{j, A} (\log T)^{j} .  $$ (To be explicit, $ V (\pm 2 x  ) =  (1 \mp i) \sqrt{x/2} \shskip \varww (x^2)   W_0 (\mp \sqrt{y} x)$.) By Lemma \ref{lem: x > MT}, the first integral is negligibly small unless $ \sqrt{\varLambda} > M^{1-\vepsilon} T $. Observe that the inner integral is a Fourier integral, and that $ \sqrt{y} + \sqrt{\varLambda} \cosh r \Gt T $ is large, so the results follow immediately.

Second, let $F_{\infty}$ be complex. Similar to the real case, one may prove \eqref{10eq: tilde w = Phi, C}   on applying \eqref{4eq: asymptotic, C} and \eqref{8eq: H-sharp(z)}, along with the substitution $\pm 2 \sqrt{z} \ra z$.  
\end{proof}

\subsection{Analysis for the Hyperbolic  Functions}

\begin{lem}\label{lem: I-}
   Let  $  \delta < \rho  \Lt 1$. For $0 \leqslant x < 1$  define the region  $\RI^- (\delta, \rho; x)$ by 
   \begin{align}\label{8eq: defn I-}
   	|r| \leqslant  \rho ,   \qquad   |\sinh r \pm  x   | \leqslant   \delta. %\text{\footnotemark}
   \end{align}  

{\rm(1)} $\RI^- (\delta, \rho; x)$ is non-empty unless $   {x} \Lt \rho  $.  

{\rm(2)} $\RI^- (\delta, \rho; x)$ has length $O (\delta)$.   
\end{lem}

\begin{proof} 
The first assertion is obvious in view of $\sinh r = O (\rho)$. By the mean value theorem, the second inequality in \eqref{8eq: defn I-} implies that  $  |r \pm \mathrm{arcsinh} \hskip 1pt   {x}  | \Lt \delta $, and hence the length of  $\RI^- (\delta, \rho; x)$ is bounded by $ O (\delta) $. 
\end{proof}

\begin{lem}\label{lem: I+} 
	
	Let  $  \sqrt{\delta} < \rho  \Lt 1$.	For  $0 < x < \sqrt{2}$  define the region  $\RI^+ (\delta, \rho; x)$  by
	\begin{align}\label{8eq: defn I+}
		|r| \leqslant  \rho , \qquad  |\cosh r - x  | \leqslant  \delta   .
	\end{align}   
%Let $x_{\oldstylenums{0}} = \sqrt{x^2-1}$.	

{\rm(1)} $\RI^+ (\delta, \rho; x)$ is non-empty unless $  \left| x -1 \right| \Lt \rho^2   $. 

{\rm(2)}  $\RI^+ (\delta, \rho; x)$ has length  $ O  (\delta  /  \sqrt{|x -1|} ) $.   

{\rm(3)} We have  $
		 \sinh r  \allowbreak \Lt \sqrt{\delta} $ on the region  $\RI^+ (\delta, \rho; 1)$. 
\end{lem}
 
\begin{proof} 
	 By $   \sinh^2 r = \cosh^2 r - 1$, the second inequality in \eqref{8eq: defn I+} implies
	\begin{align}\label{8eq: cosh, 2}
		\big|\sinh^2 r -  (x^2 - 1 ) \big| \Lt   \delta .  
	\end{align}
Then (1) and (3) are obvious. As for (2), 
	\eqref{8eq: cosh, 2} yields    $ 	|  r | \Lt \sqrt{\delta}$ if   $  \left| x -1 \right| \Lt  {\delta} $, the empty set if $ 1-x \Gt  {\delta} $, and   $ \big|r \pm \mathrm{arcsinh}   \sqrt{x^2-1} \big| \Lt \delta  /   \sqrt{x-1}$  if $    x -1   \Gt  {\delta} $  (again, by the mean value theorem), and  hence the length of  $\RI^+ (\delta, \rho; x)$ is bounded by $ O  (\delta  /   \sqrt{| x -1  |} ) $ in every case.  
\end{proof}

\subsection{Analysis for the   Trigonometric-Hyperbolic Function}

\begin{lem}\label{lem: I, C}
	
	Let  $   {\delta} < \rho  \Lt 1$.	For  $ |x| < \sqrt{2}$ and $|y| < 1$  define   $\RI  (\delta, \rho; x + i y)$ to be the set of $(r, \omega)$ such that 
	\begin{align}\label{8eq: defn I}
		|r| \leqslant  \rho ,      \qquad  |\cos \omega \cosh r - x  | \leqslant  \delta, \qquad  |\sin \omega \sinh r - y  | \leqslant  \delta  .
	\end{align}  

{\rm(1)}  $\RI  (\delta, \rho; x + i y)$ is non-empty unless $ |x| < 1 + 2 \rho $ and $|y| \Lt \rho$. 

{\rm(2)} The area of  $\RI  (\delta, \rho; x + i y)$ has bound as follows, 
 	 \begin{align}\label{8eq: bound area}
 	 \mathrm{Area} \, \RI  (\delta, \rho; x + i y)	\Lt \frac {\delta^2} {\sqrt{(|x|-1)^2 + y^2}}. 
 	 \end{align}

%{\rm(3)}   We have  $ \sin^2 \omega   \Lt \max \big\{ ||x|-1|, \rho^2 \big\}$ on  $\RI  (\delta, \rho; x+i y)$ if $|x| > 1/2$. 

{\rm(3)} We have $
  \sinh r, \sin \omega   \Lt \sqrt{\delta} $ on the region   $\RI  (\delta, \rho; \pm 1)$.  
\end{lem}

\begin{proof}
%	We shall focus on (2), since (1) is obvious while (3)  will be transparent in its proof. 
We shall focus on (2), since (1) is obvious while (3)  will be transparent in  the last case of its proof.

	By symmetry, we only need to work in the setting with $ (r, \omega) \in [0, \rho] \times [0,  \pi / 2] $ and $(x, y) \in [0, \sqrt{2}) \times [0, 1)$.  

 Consider the mapping
\begin{align}
	f : (r, \omega) \ra (\cos \omega \cosh r, \sin \omega \sinh r), 
\end{align}
so that   $\RI  (\delta, \rho; x + i y)$ is contained in  the preimage  under $f$ of the square with center $(x, y)$ and  area $4\delta^2$.
The  Jacobian matrix 
\begin{align*}
	J_f (r, \omega) =   \begin{pmatrix}
		\, \cos \omega \sinh r &  \sin \omega \cosh r \\
		- \sin \omega \cosh r & \cos \omega \sinh r
	\end{pmatrix}  . 
\end{align*} 
On the semi-closed rectangle $(0 , \rho] \times (0, \pi/2)$, since all the principal minors of $J_f  (r, \omega)$ are positive, by the Univalence Theorem of Gale and Nikait\^o  (\cite[\S \S 4.2, 4.3]{Gale-Nikaido}), $f$ is a  univalent mapping. Note that the Jacobian determinant is equal to $\sinh^2 r + \sin^2 \omega $. Therefore  $f $ may be used as a coordinate transform, and if we are able to prove  the lower bound 
\begin{align} \label{8eq: lower bound}
		\sinh^2 r + \sin^2 \omega \Gt \sqrt{(x-1)^2 + y^2 } 
\end{align} 
on  $  \RI  (\delta, \rho; x + i y)$ for either    $ |x - 1| \Gt \delta $ or $y \Gt \delta$, then \eqref{8eq: bound area} follows immediately in this case.

Now we prove \eqref{8eq: lower bound}. For $ x \leqslant 1/2$, say, the second inequality in \eqref{8eq: defn I} implies $ \cos \omega \leqslant 1/\sqrt{2}$ (provided that $\rho \Lt 1$, so that $\cosh r$ is near $1$ and $\delta < \rho$ is  small), and hence \eqref{8eq: lower bound} is clear. For $x > 1/2$, observe that the second inequality in \eqref{8eq: defn I} implies 
\begin{align}\label{8eq: cos cosh, 2}
	\big|\sinh^2 r - \sin^2 \omega -  \sin^2 \omega \sinh^2 r   - (x^2-1)  \big| \Lt   \delta , 
\end{align}
due to $ \cos^2 \omega \cosh^2 r = 1 + \sinh^2 r - \sin^2 \omega -  \sin^2 \omega \sinh^2 r $. %Note that (3)  follows immediately from \eqref{8eq: cos cosh, 2}.  
In the case when   $ |x - 1| \Gt \delta $ and $y \Gt \delta$, the last inequality in  \eqref{8eq: defn I} and \eqref{8eq: cos cosh, 2} together yield 
\begin{align*}
\sinh^2 r - \sin^2 \omega \sasymp x^2-1, \qquad \sin \omega \sinh r \sasymp y ,
\end{align*}
and hence \eqref{8eq: lower bound} by $ \sinh^2 r + \sin^2 \omega = \sqrt{\big(\sinh^2 r - \sin^2 \omega\big)^2 + 4 \sin^2 \omega \sinh^2 r } $. The proof is similar for the remaining two cases when $ |x - 1| \Lt \delta $ or $y \Lt \delta$. 

Finally, in the case when $ |x - 1| \Lt \delta $ and $y \Lt \delta$, we have $|\cos \omega \cosh r - 1| \Lt \delta$ and $|\sin \omega \sinh r    | \Lt  \delta$ (so the Jacobian of $f$ could be very small or vanish). Since $(\cosh r - \cos \omega )^2 = (\cos \omega \cosh r - 1)^2 + (\sin \omega \sinh r)^2 $ and $ \cosh^2 r - \cos^2 \omega = \sin^2 \omega + \sinh^2 r$, it follows that the area of  $\RI  (\delta, \rho; x + i y)$ is bounded by $O (\delta)$, and hence \eqref{8eq: bound area}. Moreover, (3) is also clear from these arguments.
\end{proof}

In practice $z = \sqrt{n/m}$ ($m, n \in \frO' \smallsetminus \{0\}$). The simple lemma below will help us take care of the square root in the complex case, with (1)--(4) corresponding to \eqref{12eq: -}--\eqref{12eq: O+, 2} in \S \ref{sec: off, C}. 

\begin{lem}\label{lem: square root}
%	Let $x = \sqrt{x'}$. 
	
%	{\rm(1)} If $ |x-1| < \rho $, then $|x'-1| \Lt \rho$, and $|x'-1| \asymp |x-1|$. 
	
%	\noindent	
Write $z = x+iy$ and $z^2 = x_2+iy_2$.  Let $ y  \Lt \rho$.  
	
	{\rm(1)} If $ |x| \Lt \rho $, then $ |z^2  | \Lt \rho^2 $. 
	
	{\rm(2)} If $|x| \Gt \rho $, then $x_2 \asymp x^2$ and $y_2 \Lt \rho |x| $.
	
	{\rm(3)} If $ ||x|-1| \Lt \rho  $, then $ |z^2-1| \Lt \rho $  and $|z^2-1|^2 \asymp (|x|-1)^2 + y^2$. 
	
	{\rm(4)} If $1 - |x|  \Gt \rho   $, then $|x_2-1| \asymp 1-|x|$ and $y_2 \Lt \rho$.

\end{lem}

\subsection{Estimates for the $\Psi$-integrals} \label{sec: estimates for Phi}

Let 
\begin{align}\label{8eq: rho and delta}
	  \rho = M^{\vepsilon}/ M, \qquad \delta = T^{\vepsilon} /  {\varDelta} . 
\end{align}
It is then clear that the $\Psi$-integrals $\Psi^{\pm} (x, \varDelta)$ and  $\Psi  (z, \varDelta)$ defined in Lemma \ref{lem: Hankel} are trivially bounded by the area  of $\RI^{\pm} ( \delta  , \rho; x) $ and  $\RI  ( \delta  , \rho; z) $ respectively. A direct consequence of Lemma \ref{lem: I-}, \ref{lem: I+}, and \ref{lem: I, C} is the following proposition. For brevity, we shall allow  $M^{\vepsilon}$ to absorb absolute constants---for example,   the factor $2$ in $|x| < 1+ 2 \rho$ and  the implied constant in $|y| \Lt \rho$ ($\rho = M^{\vepsilon}/ M$). 

\begin{prop}\label{lem: bounds for Phi, R}
Let  $ \Psi^{\pm} (x, \varDelta) $ and $\Psi  (z, \varDelta)$ be   as in {\rm\eqref{10eq: Phi+ (x), R}}, {\rm\eqref{10eq: Phi- (x), R}} and {\rm\eqref{10eq: Phi (x), C}}.	 
	  
	 {\rm(1)}  $\Psi^{-} (x, \varDelta)$ or $\Psi^{+} (x, \varDelta)$     is negligibly small unless   $x < M^{\vepsilon}/ M$ or  $|x-1| < M^{\vepsilon}/ M^2 $   respectively,   in which case 
	 \begin{align}
	 	\Psi^{-} (x, \varDelta)  \Lt \frac {T^{\vepsilon}} {\varDelta  },  \qquad \Psi^{+} (x, \varDelta) \Lt \frac {T^{\vepsilon}} {\varDelta \sqrt{|x-1|}}  .
	 \end{align}

{\rm(2)}  $\Psi  (x+i y, \varDelta) $ is negligibly small unless    $ |x| < 1 + M^{\vepsilon}/M  $ and $|y| < M^{\vepsilon}/M$, in which case   
\begin{align}
	\Psi  (x + iy , \varDelta) \Lt \frac {T^{\vepsilon}} {\varDelta^2   \sqrt{(|x|-1)^2 + y^2}   } .
\end{align} 
\end{prop}

%The estimates $T^{\vepsilon}/\varDelta$ and  $T^{\vepsilon}/\varDelta^2$ as above are adequate  for $\Psi^{-} (x, \varDelta)$ and $\Psi  (x+i y, \varDelta) $, with $ |x|  \leqslant 1/2 $,  respectively. 
%Yet we have not taken advantage of the exponential factor $e (Tr/\pi)$. By recourse to partial integration for the $r$-integral, we now prove that  $\Psi^{+} (x, \varLambda)$ and $\Psi  (x+i y, \varDelta) $ are negligibly small under certain circumstances. 

%However, the estimates above are not sufficient for $ \Psi^{+} (1, \varDelta)$ and $\Psi (\pm 1, \varDelta)$, which will arise in the  diagonal frequency. Fortunately, 

Finally, by recourse to partial integration for the $r$-integral, we    prove that  $ \Psi^{+} (1, \varDelta)$ and $\Psi (\pm 1, \varDelta)$ are negligibly small for $\varDelta \leqslant T^{2-\vepsilon}$. 

\begin{prop}\label{prop: small Phi(1)} Let  $ \Psi^{+} (x, \varDelta) $ and $\Psi  (z, \varDelta)$ be defined  as in {\rm\eqref{10eq: Phi+ (x), R}} and {\rm\eqref{10eq: Phi (x), C}}.
	
	{\rm(1)} 
	We have $\Psi^{+} (1, \varDelta)  = O_{A, \vepsilon}  (T^{-A})$ if $\varDelta \leqslant T^{2-\vepsilon}$. 
	
	 {\rm(2)} We have $  \Psi  (\pm 1, \varDelta) = O_{A, \vepsilon}  (T^{-A})$ if $\varDelta \leqslant T^{2-\vepsilon}$. 
\end{prop}

\begin{proof}
%Let $\rho$ and $\delta$ be as in \eqref{8eq: rho and delta}. 	
There are three steps.  First,   smoothly truncate the $r$-integral to the range $|r| \leqslant \rho$. Second, repeat partial integration.  Fa\`a di Bruno's formula (\cite{Faa-di-Bruno}) and its extension are required to calculate the higher $r$-derivatives of  $  \widehat{V}  ( \varDelta  (x - \cosh r )   )$ and $\widehat{V}  ( \varDelta  ( z -  \trh (r, \omega)    )   )$. Third,   confine the integration to  the region $\RI^{+} (\delta, \rho; 1) $  or  $\RI (\delta, \rho; \pm 1) $, and use the bounds for $\sinh r$ or $\sin \omega$ in Lemma \ref{lem: I+}  (3) or Lemma \ref{lem: I, C} (3), respectively. In this way, one obtains   high powers of $ \varDelta \sqrt{\delta}  / T = \sqrt{\varDelta} / T^{1-\vepsilon}$.   The details are left to the   readers. 
\end{proof}

\subsection{Remarks on the Complex Case} The results in the complex case may be improved when $ x $ is close to $\pm 1$, in correspondence to the case of $\Psi^+ (x, \varDelta)$. However, the improvements will not be useful, since the worst case scenario is when $x$ stays away from $0$ and $\pm 1$, say around $1/2$. See \S \ref{sec: off, C}.

\section{Mellin Transform  of Bessel Kernels}\label{sec: Mellin}

In this section, we derive explicit formulae for the Mellin transform  of the Bessel kernel $B_{it} (x)$ and $B_{it} (z)$.  To be precise, define
\begin{align}\label{9eq: Mellin, R, 0}
\widetilde{B}_{it} (s) =	\int   B_{it} (x  ) |x|^{  s - 1}  { \nd x },
\end{align}
or
\begin{align}\label{9eq: Mellin, C, 0}
	\widetilde{B}_{it} (s) = \iint   B_{it} (z  ) |z|^{ 2 s - 2}  { \nd z },
\end{align}
according as $F_{\infty}$ is real or complex. 
%It suffices to consider $  M_{\rho, s} (1)$ since 
%\begin{align}
%	M_{\rho, s} (y) = |y|^{-  \rho} M_{\rho, s} (1)  , \qquad M_{\rho, s} (u) = |u|^{ - 2 \rho } M_{\rho, s} (1)   . 
%\end{align} 

\begin{lem}\label{lem: Mellin}
	 For $ |\mathrm{Im} (t)| < \mathrm{Re} (s) < \frac 1 4  $  the Mellin integral $\widetilde{B}_{it} (s)$ in {\rm\eqref{9eq: Mellin, R, 0}} or {\rm\eqref{9eq: Mellin, C, 0}} is absolutely convergent, and 
	 \begin{align}\label{9eq: Mellin=gamma}
	 	\widetilde{B}_{it} (s) = \frac{\gamma (s, t)}{\gamma (1-s, t)},
	 \end{align}
 with $\gamma (s, t)$ defined in {\rm\eqref{4eq: defn of gamma (s, f)}}.
\end{lem}

\begin{proof}
For $ |\mathrm{Im} (t)| < \frac 1 4 $ we have crude estimates:
\begin{align*}%\label{9eq: crude bounds for Bs}
	B_{it} (x) \hskip -1 pt \Lt_{t, \vepsilon} \hskip -2 pt \min \bigg\{ \hskip -1 pt \frac 1  {|x|^{ |\mathrm{Im} (t)| + \vepsilon}}, \frac 1 {\sqrt[4]{|x|}} \hskip -1 pt \bigg\}, \quad B_{it} (z) \hskip -1 pt \Lt_{t, \vepsilon} \hskip -2 pt \min \bigg\{ \hskip -1 pt \frac 1  {|z|^{ |\mathrm{Im} (2t)| + \vepsilon}}, \frac 1 {\sqrt{|z| }} \hskip -1 pt \bigg\}  , 
\end{align*} 
so the convergence of integrals is clear.

For the real case, by \cite[\S 7.7.3 (19), (27)]{ET-II}, along with  Euler's reflection formula, we have
\begin{align*}
	\int_0^{\infty} J_{\mu} (4 \pi x) x^{\rho-1} \nd x = \frac { 1 } {  (2\pi)^{\rho+1}  } \sin \bigg( \frac { \pi (\rho - \mu)  } 2  \bigg) \Gamma \bigg(\frac { \rho + \mu  } 2 \bigg)  \Gamma  \bigg(\frac { \rho - \mu  } 2 \bigg) ,
\end{align*}
for $- \mathrm{Re} (\mu) < \mathrm{Re} (\rho) < \tfrac 1 2$, and
\begin{align*}
	\int_0^{\infty} K_{\mu} (4 \pi x) x^{\rho-1} \nd x = \frac {  1 } {  4 (2\pi)^{\rho}  } \Gamma \bigg(\frac { \rho + \mu  } 2 \bigg)  \Gamma  \bigg(\frac { \rho - \mu  } 2 \bigg), 
\end{align*} 
for $|\mathrm{Re} (\mu)| <  \mathrm{Re} (\rho  )$.

For the complex case, we have
\begin{equation*} 
		  \int_{0}^{2 \pi} \int_0^\infty \boldJ_{ \mu }     (  x e^{i\phi}  )   x^{2 \rho - 1}  \nd x \shskip \nd \phi 
		=    \frac {     \cos (\pi \mu) - \cos (\pi   \rho) } {  (2 \pi)^{ 2 \rho + 2 }  } \Gamma \bigg(\frac { \rho + \mu  } 2 \bigg)^2 \Gamma  \bigg(\frac { \rho - \mu  } 2 \bigg)  ^2 .  
\end{equation*}
for  $ |\mathrm{Re} (\mu)| < \mathrm{Re} (\rho) < \frac 1 2  $, with 
\begin{equation*}%\label{0eq: defn of Bessel}
	\boldsymbol{J}_{ \mu } (z)   =  \frac {1} {\sin (\pi \mu)} \big( J_{-\mu   } (4 \pi  z) J_{-\mu   } (4 \pi \widebar z) -  J_{ \mu   } (4 \pi  z) J_{ \mu   } (4 \pi \widebar z) \big) .
\end{equation*} 
This is a simple consequence of Theorem 1.1 and Proposition 3.2 in \cite{Qi-BE}, specialized to the case  $d = 0$ and $y = 0$. Note that Gauss' hypergeometric function is equal $1$ at the origin. 

In view of  Definition \ref{defn: Bessel kernel},  one   derives 
\begin{align*}%\label{9eq: Mellin, R}
	\int   B_{it} (x  ) |x|^{  s - 1}  { \nd x }  & = \frac { 2 \lp \cos (\pi  it) + \cos (\pi  s) \rp } {   (2\pi   )^{2s}   } \Gamma  (    s + it   ) \Gamma  (     s - it  ) , \\
	%\label{9eq: Mellin, C}
	\iint   B_{it} (z  ) |z|^{ 2 s - 2}  { \nd z } & = \frac {  2  \lp \cos (2\pi it) - \cos (2\pi   s) \rp } {       (2 \pi    )^{ 4 s   }  } \Gamma    ( s +   it  ) ^2 \Gamma  ( s  -   it  )^2.
\end{align*}
Then \eqref{9eq: Mellin=gamma} readily follows from Euler's reflection formula and Legendre's duplication formula (the latter is needed only for the real case).  
\end{proof}

\begin{rem}
	The formula {\rm\eqref{9eq: Mellin=gamma}} can also be interpreted from the view point of representation theory for local functional equations. See {\rm\cite[\S 17]{Qi-Bessel}}. 
\end{rem}

{ \large \part{The Twisted First and Second Moments}}

\section{Setup: Application of  the Kuznetsov Formula}

Now we turn to the investigation of the   twisted first and second moments:
\begin{align}
	\SM_q (\frm ) =  \sumh_{f \in  \SB  } \hskip -1pt    k  ( t_f ) \lambda_f  ( \frm )    L \big( \tfrac 1 2 , f \big)^q    
\end{align} 
for $q = 1$ or $2$, and weight function $ k (t) $ defined as in \eqref{1eq: defn of kq} or \eqref{5eq: defn k(nu)}. In the sequel, we shall always let $\frm = m \frD$. 

By the    Approximate Functional Equations \eqref{5eq: AFE, 1} and \eqref{5eq: AFE, 2},    we infer that
\begin{align}
	\SM_q (\frm ) = 2   \sum_{\frn \shskip \subset \shskip \frO }  \frac {   \tau (\frn)^{q-1} } { \sqrt{\RN  (  \frn   )}  }      \sumh_{f \in  \SB  } \hskip -1pt    k  ( t_f ) \lambda_f  ( \frm ) \lambda_f  (\frn ) V_q   (  \RN   (  \frn  \frD^{-q}  ); t_f   )    .
\end{align}  
In view of  \eqref{1eq: derivatives for V(y, t), 1} in Lemma \ref{lem: afq} (1), at the cost of a negligible error term, we may truncate the summations over $\frn$ to the range $ \RN (\frn) \leqslant T^{  q N  + \vepsilon}$.  

Next, we use the expressions  of $V_q  \lp  \RN   (  \frn  \frD^{-q}  ); t    \rp$  as in \eqref{1eq: approx of V1} and \eqref{1eq: approx of V2} in Lemma \ref{lem: afq} (1) with $U = \log T$ (so that the errors therein are negligible), and then apply the Kuznetsov trace formula in Proposition \ref{prop: Kuznetsov} inside the $v$-integral with test function:  
\begin{align} \label{9eq: h (t; v)}
h^q (t; v) =	k (t) G (v, t)^q ;
\end{align} 
see \eqref{5eq: defn h(nu)}.   Moreover, for the diagonal   and the Eisenstein  contributions, with the loss of     negligible errors, we revert the $v$-integral to $V_q  (  \RN   (  \frn  \frD^{-q}  ); t    )$, and for the latter convert the $\frn$-sum to $ \left| \zeta_F \big(\tfrac 1 2 + it \big) \right|^{2q} $ by the  Approximate Functional Equations \eqref{5eq: AFE zeta, 1} and \eqref{5eq: AFE zeta, 2}. 
It follows that 
\begin{align}
	\SM_q (\frm )   =   \SD_q (\frm ) - \SE_q (\frm ) + \SO_q (\frm ) +  O \lp T^{-A} \rp,   
\end{align}
where $\SD_q (\frm ) $ is the diagonal term (it exists when $\RN (\frm) \leqslant T^{q N + \vepsilon}$)
\begin{align}\label{9eq: Dp(m)}
	  \SD_q (\frm ) = 2 c_1 \frac { \tau (\frm)^{q-1}} {\sqrt{\RN(\frm)}}\SDH_q (\frm  ),
\end{align}
with
\begin{align}\label{9eq: Hp(m)}
\SDH_q (\frm ) = \int_{-\infty}^{\infty} k (t) V_q   (  \RN   (  \frm  \frD^{-q}  ); t    ) \nd \shskip \mu (t) , 
\end{align}
$\SE_q (\frm )$ is the   Eisenstein (continuous spectrum)  term 
\begin{align}\label{9eq: Ep(m)}
	\SE_q (\frm ) =   \frac {1}  {4\pi} c_0  \int_{-\infty}^{\infty}    k (t) \tau_{it} (\frm ) \omega (t)   \left| \zeta_F \big(\tfrac 1 2 + it \big) \right|^{2q} 
	 \shskip   \nd t, 
\end{align}
and $\SO_q (\frm )$ is the off-diagonal term
\begin{align}\label{9eq: O1(m)}
\SO_1 (\frm) & =	\frac {2  } {    \pi i   } \frac {c_2} {\sqrt{|d_F|}} \int_{  \shskip \vepsilon - i \log T}^{\vepsilon + i \log T}    \SO_1 (\frm; v ) \frac {\nd v} {v}, \\
\label{9eq: O2(m)}\SO_2 (\frm) & =	\frac {2  } {    \pi i   } \frac {c_2} {\sqrt{|d_F|}} \int_{  \shskip \vepsilon - i \log T}^{\vepsilon + i \log T}    \SO_2 (\frm; v ) \zeta (1+2v) |d_F|^{ v} \frac {\nd v} {v}, 
\end{align} 
with 
\begin{align}\label{8eq: O (m), 0} 
		\SO_q (\frm; v ) =  \sum_{(c)   \subset \shskip \frO  } \frac { 1 } { |\RN (c  )| }  \mathop{ \sum_{n \shskip \in \shskip \frO'   } }_{ |\RN(n  )| \shskip \leqslant T^{q N  +\vepsilon}  } \frac {   \tau (n   \frD)^{q-1} } {  {|\RN  (  n      )|^{1/2+v}}  }   {S ( m,   n   ; c  ) }   \SDH_q \bigg( \frac { mn } {  c^2    } ; v \bigg),  
\end{align}  
and
\begin{align}\label{9eq: Hp(x)}
	\SDH_q (x; v) = \int_{-\infty}^{\infty} h^q (t; v) B_{i t} (x ) \nd \shskip \mu (t) . 
\end{align}
Note that the factor $2$ arises in \eqref{9eq: O1(m)} and \eqref{9eq: O2(m)} when we combine the $\epsilon$- and $\frn$-sums into an $n$-sum,  and   fold the $c$-sum into a  $(c)$-sum over ideals. 

%\vskip 5pt

%In the next two sections, we shall  derive the following asymptotic formulae for  $\SM_1 (\frm )$ and $\SM_2 (\frm )$ in Theorem \ref{thm: moment}. 

\section{The Twisted First Moment} \label{sec: 1st moment}

%In this section, we prove the asymptotic formula \eqref{10eq: 1st moment} for $ \SM_1 (\frm ) $. 

%\subsection{The Diagonal Term} 

Let us first treat the diagonal term $ \SD_1 (\frm) $ as defined by \eqref{9eq: Dp(m)} and \eqref{9eq: Hp(m)}. It contains the main term for $ \SM_1 (\frm )$.

Recall  the definitions of  $\nd \mu (t)$ and     $  k (t) $ given by \eqref{1eq: defn Plancherel measure} and \eqref{5eq: defn k(nu)}. Now we apply  \eqref{4eq: asymptotic for V1}  in Lemma   \ref{lem: afq} (1) to analyze $ \SDH_1 (\frm)$. The main term   yields
\begin{align*}
	\int_{-\infty}^{\infty} k (t)   \nd \shskip \mu (t) = 2 \sqrt{\pi} M T^{N}     \big(  1 + O  \big( (M/T)^2 \big)  \big), 
\end{align*}
which can be easily seen by truncation  near $t = \pm T$ and the change of variable $t \ra M t \pm T$.   The error-term contribution is bounded by  $   (  \RN (\frm) / T^N )^{A }$  and hence  negligibly small if $\RN (\frm) \leqslant T^{N-\vepsilon}$ and $A $ is large in terms of $\vepsilon$. 
We conclude that 
\begin{align}\label{10eq: diagonal, D1}
	\SD_1 (\frm ) =  4 \sqrt{\pi} c_1 \frac { M T^N    } {\sqrt{\RN(\frm)}}    \big(  1 + O_{\vepsilon}  \big( (M/T)^2 \big)   \big) . 
\end{align}

%\subsection{The Eisenstein Term} 

For the Eisenstein term $\SE_1 (\frm)$, on inserting \eqref{2eq: subconvex} and \eqref{3eq: bound for omega(t)} into \eqref{9eq: Ep(m)} and estimating the integral trivially, we obtain
\begin{align}\label{10eq: bound for E1(m)}
	\SE_1 (\frm) = O_{\vepsilon} \big(  M T^{2 N \theta + \vepsilon}  \big) . 
\end{align}
However,   \eqref{10eq: bound for E1(m)} may be improved into 
\begin{align}\label{10eq: bound for E1(m), 2}
	\SE_1 (\frm) = O_{\vepsilon}  (  M T^{ \vepsilon}   ),
\end{align} 
if $M \geqslant T^{\frac {1273} {4053}+\vepsilon}$ for $F = \BQ$ or $M \geqslant T^{\frac 7 8 + \vepsilon } $ for $F = \BQ (\sqrt{d_F})$. For this use the estimate for  the second
moment of $\zeta  \big(\frac 1 2 + it \big)$ on short intervals in \cite[Theorem 3]{BW-Riemann-2}  or  the asymptotic formula  for the second moment of $\zeta_F \big(\frac 1 2 + it \big)$ in \cite{Muller-Dedekind-Quadratic}. 

%\subsection{The Off-diagonal Term} 
Finally, we consider the off-diagonal term $\SO_1 (\frm  )$  given by  \eqref{9eq: O1(m)},  \eqref{8eq: O (m), 0}, and \eqref{9eq: Hp(x)}.  Since   $|\RN( m  )| \shskip \leqslant T^{N  - \vepsilon}$ and  $|\RN(n  )| \shskip \leqslant T^{N  +\vepsilon}$, one may adjust $\vepsilon$ so that $ \left|m n/c^2 \right| \Lt T^2$, and Lemma \ref{lem: H(x), |z|>1} (3)--(6) implies that $\SDH_1  (   { mn } / {  c^2    } ; v  )$,  $\SO_1 (\frm; v )$, and hence   $\SO_1 (\frm  )$ are negligibly small. A remark  is that Weil's bound for $S (m, n; c)$ is needed (one could use $O \big(\sqrt{\RN(c\frm)}\tau (c) \big)$) to ensure that the $(c)$-sum is convergent. 

In conclusion, the asymptotic formula \eqref{10eq: 1st moment} in Theorem \ref{thm: moment} is established on the foregoing arguments.

\section{The Twisted Second Moment}\label{sec: 2nd moment}

This section is devoted to the proof of the asymptotic formula \eqref{10eq: 2nd moment} for $ \SM_2 (\frm)$  in Theorem \ref{thm: moment}. %Our focus will be on the off-diagonal term $ \SO_2 (\frm ) $.

%\subsection{The Diagonal Term}

The analysis of $\SD_2 (\frm)$, albeit slightly more involved,  is  similar to that of $\SD_1 (\frm)$. By \eqref{4eq: asymptotic for V2} and \eqref{4eq: asymp of psi} in Lemma   \ref{lem: afq} (2),   $\SDH_2 (\frm)$ is equal to
\begin{align*}
	 \int_{-\infty}^{\infty} k (t) \Big( \gamma_{1} \Big( N \log {\textstyle \sqrt{\tfrac 1 4 + t^2}} - \log   {\sqrt{\RN(\frm)}} \Big) + \gamma_0 ' \Big) \nd \shskip \mu (t) + O _{\vepsilon}  \big(M T^{N-2}  \big),  % \big(  \gamma_{0} + \gamma_{1} \big( N \log \lp \RC (t) / 2\pi \rp  -  \log  \big( \sqrt{\RN(\frm)} / |d_F|\big)  \big)\big)  \nd \shskip \mu (t) + O  \big(M T^{N-1}  \big),
\end{align*}
with $\gamma_0'$ defined as in Theorem \ref{thm: moment}. Consequently, 
\begin{align}\label{10eq: diagonal, D2}
	\SD_2 (\frm ) = 4 \sqrt{\pi} c_1 \frac { \tau (\frm) M T^N  } {\sqrt{\RN(\frm)}} \bigg(   \gamma_{1} \log \frac {T^N} {\sqrt{\RN(\frm)}} + \gamma_0 '      + O_{\vepsilon}  \big(   (M/T)^2 \log T \big)  \bigg) . 
\end{align}
It should be stressed that  $ \SD_2 (\frm) $ only contributes {\it half}   the main term for $ \SM_2 (\frm )$.

%\subsection{The Eisenstein Term} 
The trivial estimate for $\SE_2 (\frm)$ obtained from  \eqref{2eq: subconvex} and \eqref{3eq: bound for omega(t)} is as follows:
\begin{align}\label{10eq: bound for E2(m)}
	\SE_2 (\frm) = O_{\vepsilon} \big( M T^{4 N \theta + \vepsilon}  \big) . 
\end{align}
By \eqref{2eq: subconvex} and \eqref{10eq: bound for E1(m), 2}, we improve \eqref{10eq: bound for E2(m)}  into 
\begin{align}\label{10eq: bound for E2(m), 2}
	\SE_2 (\frm) = O_{\vepsilon} \big( M T^{2 N \theta + \vepsilon}  \big) , 
\end{align} 
for  $M \geqslant T^{\frac {1273} {4053}+\vepsilon}$  or $M \geqslant T^{\frac 7 8 + \vepsilon } $ according as $F = \BQ$ or $  \BQ(\sqrt{d_F})$. 
Further, if $F = \BQ$, then  \eqref{10eq: bound for E2(m), 2} may be improved into
\begin{align}\label{10eq: bound for E2(m), Q}
	\SE_2 (m) = O_{\vepsilon} \big(  M T^{    \vepsilon}  \big) 
\end{align} 
for $M \geqslant T^{\frac 2 3 + \vepsilon } $, by the estimate   for the fourth  
moment  of $\zeta  \big(\frac 1 2 + it \big)$ on short intervals in \cite[\S 6]{Ivic-Riemann-4} (see also \cite{IM-4th-Moment}).  As for the fourth moment of $ \zeta_F \big(\frac 1 2 + it \big) $ for $F = \BQ (\sqrt{d_F})$, an explicit spectral formula is known over the Gaussian field in \cite{B-Mo} but currently we  do not know how it can be used to obtain non-trivial estimate (asymptotic is beyond our reach as $ | \zeta_F (s) |^4$ is of degree $8$). 

%\subsection{The Off-diagonal Term}

%This section is devoted to   the study of  the off-diagonal term $ \SO_2 (\frm ) $  given by  \eqref{9eq: Op(m)} and \eqref{8eq: O (m)}. Given the analysis in Part \ref{part: analysis}, the reader is suggested to work out the case $F = \BQ$ as an exercise prior to reading this section. For the general case, although the strategy is the same, the notation becomes more complicated, especially at the end if there is a complex place. 
 
Now we turn to the study of  the off-diagonal term $ \SO_2 (\frm ) $ (see  \eqref{9eq: O2(m)}--\eqref{9eq: Hp(x)}). 

First of all, by Lemma \ref{lem: H(x), |z|>1} (3)--(6), one may  impose the condition $      \big| m n       / c^2 \big| \Gt  T^{2 } $ to the summations, with the cost of a negligible error.  
Let   $\sum_{R} \varvv  (|x| / R ) $ be a dyadic partition of unity for $\Fx_{\infty}$, with  $R  = 2^{j  / 2}$  and $ \varvv  (r) \in C_c^{\infty} [1, 2]$. It may be exploited to partition the sum in \eqref{8eq: O (m), 0}   into $O  ( \log T  )$ many sums of the form 
\begin{align}\label{12eq: a-sum}
	\begin{aligned}
		\SO_2 (\frm;  R ; v ) =  \frac 1 {R^{N/2+N v}} & \mathop{\sum_{(c)   \subset \shskip \frO  }}_{|c|  \Lt \sqrt{|m| R}/ T} \frac { 1 } { |\RN (c  )| } \\
	& \cdot \sum_{n \shskip \in \shskip \frO' \smallsetminus \{0\}   }     {   \tau (  n \frD)  }    S ( m,   n   ; c  )  \varww \lp \frac n R   ,  \frac {m R}   {c^2} ; v \rp ,
	\end{aligned}
\end{align} 
for   $R  \leqslant T^{2 +\vepsilon}$, where 
\begin{align*}
	  \varww \lp x   ,   \varLambda ; v \rp  =   {\varww (|x|; v) \SDH_2  ( \varLambda x ; v  ) }   , \qquad \varww (r; v) = \frac{\varvv (r)   } {r^{N/2 + N v}}. %, \quad \SDH  (   x  ) = \SDH_2  (   x; v  ) .
\end{align*}
Clearly, the weight function $\varww \lp x   ,   \varLambda; v \rp$ is of the form in \eqref{11eq: defn of w (x, Lmabda), R} or \eqref{11eq: defn of w (z, Lmabda), C}. Note that  $ \varww^{(j)} (r; v)  \Lt_{j } (\log T)^{j} $ holds  
uniformly  for $v \in [\vepsilon - i \log T, \vepsilon + i \log T]$.

\subsection{Application of the  Vorono\"i   Summation}

Next, in \eqref{12eq: a-sum} we open the Kloosterman sum $S ( m,   n   ; c  )$ (as in \eqref{2eq: defn Kloosterman KS}) and apply the Vorono\"i summation formula (see \eqref{app: Voronoi, tau} and \eqref{3eq: limit for 0}) to the $n$-sum. It is clear that the exponential sum over $(\frO / c \hskip 1pt \frO)^{\times}$ turns into  the Ramanujan sum $S (m-n, 0; c)$. 

For the entire zero-frequency contribution, % including the $(c)$-summation and $v$-integration, 
we reverse the procedures above---truncation  and   partition of unity---and shift the integral contour for $v$ to $\mathrm{Re} (v) = \frac 1 3$, costing only  negligible errors. We obtain
 \begin{align}\label{12eq: 0-frequency}
 	\SZ  (\frm) = \frac {2  } {    \pi i    } c_2   \int_{-\infty}^{\infty}    k (t)  \int_{ (\frac 1 3) }    G (v, t)^2   \zeta (1+2v)  \widetilde Z   (m; v, t) \frac {\nd v} {v}  \nd \shskip \mu (t) ,
 \end{align}
where
\begin{align}\label{12eq: Z}
\widetilde Z   (m; v, t ) \hskip -1 pt = \lim_{\delta \ra 0}	\sum_{\pm} {\zeta_F (1\pm 2 \delta)   } |   d_F  | ^{v \pm  \delta} \hskip -2 pt \sum_{(c)   \subset \shskip \frO  } \hskip -2 pt \frac { S (m ,  0; c)  } { |\RN (c  )|^{2 \pm 2\delta} }      \widetilde {B}_{ it} \big(m/c^2; \tfrac 1 2 -v \pm \delta \big), 
\end{align} 
and 
\begin{align}\label{12eq: Mellin}
\widetilde {B}_{it} (y; s) = \int_{\Fx_{\infty}}  B_{it} (x y ) \|x\|_{\infty}^{  s - 1}  { \nd x }. 
\end{align}
Note that we can effectively truncate the $t$-integral near $\pm T$ and the $v$-integral at height $\log T$, that the $(c)$-sum and the $x$-integral are absolutely convergent (see the proof of Lemma \ref{lem: Mellin}), and that the expression in the limit is analytic in the  $\delta$-variable. At any rate,  it is legitimate to arrange the order of sums and integrals in the above manner. 

The next lemma manifests that $\SZ (\frm )$ contributes the other {\it half} of the main term for $\SM_2 (\frm)$. Compare \eqref{10eq: diagonal, D2}.

\begin{lem}\label{lem: zero}
	We have 
	\begin{align}\label{10eq: zero}
		\SZ (\frm ) = 4 \sqrt{\pi} c_1 \frac { \tau (\frm) M T^N  } {\sqrt{\RN(\frm)}} \bigg(   \gamma_{1} \log \frac {T^N} {\sqrt{\RN(\frm)}} + \gamma_0 '      + O_{\vepsilon}  \big(   (M/T)^2 \log T \big)   \bigg) . 
	\end{align} 
\end{lem}

For the dual sum, it remains to prove the following estimates. For brevity, we have suppressed $v$ from our notation. 

\begin{lem}\label{lem; bound for dual}
Let $R \leqslant T^{2+\vepsilon}$. Let   $\varww  (r) \in C_c^{\infty} [1, 2] $  satisfy $ \varww^{(j)} (r)  \Lt_{j} (\log T)^{j} $. 	Define 
	\begin{align}\label{12eq: a-sum, 2}
		\begin{aligned}
			\widetilde{\SO}_2 (\frm;  R ) =   \hskip -3 pt 
			\mathop{\sum_{(c)   \subset \shskip \frO  }}_{|c|  \Lt \sqrt{|m| R}/ T} \hskip -3 pt  \frac { 1 } { |\RN (c  )|^2 }  \hskip -2 pt  \sum_{n \shskip \in \shskip \frO' \smallsetminus \{0\}   }  \hskip -2 pt    {   \tau (  n \frD)  }    S ( m - n , 0  ; c  )  \widetilde \varww_0 \lp \frac {n R} {c^2}   ,  \frac {m R}   {c^2}   \rp ,
		\end{aligned}
	\end{align}  
with 
\begin{align}
\varww \lp x   ,   \varLambda   \rp  =   {\varww (|x| ) \SDH_2  ( \varLambda x   ) }, \qquad 	\widetilde {\varww}_0 (y , \varLambda ) = \int_{F_{\infty}} {\varww} (x , \varLambda ) B_0 (xy) \nd x . 
\end{align}
Then 
\begin{align}\label{12eq: bound for O, Q}
\sqrt{R }\,	\widetilde{\SO}_2 (m ;  R ) \Lt \left\{ \begin{aligned}
	& T^{-A}, & & \text{ if }  0 < m \leqslant M^{2-\vepsilon}, \\
	& \frac {\sqrt{ m } T^{1/2+\vepsilon}} {\sqrt{M}}, & & \text{ if } M^{2-\vepsilon} < m \leqslant T^{2-\vepsilon}, 
\end{aligned}\right.  
\end{align}
for $F = \BQ$, and
\begin{align}\label{12eq: bound for O, C}
	R \,	\widetilde{\SO}_2 (\frm;  R ) \Lt   {\sqrt{\RN(\frm) } T^{1+\vepsilon}} + \frac {M^2 T^{1+\vepsilon}} {\sqrt{\RN(\frm)}},
\end{align}
for $F = \BQ (\sqrt{d_F})$. 
\end{lem}

The asymptotic formula in \eqref{10eq: 2nd moment} now follows by combining \eqref{10eq: diagonal, D2}--\eqref{10eq: bound for E2(m), Q}, \eqref{10eq: zero}, \eqref{12eq: bound for O, Q}, and \eqref{12eq: bound for O, C}.

\subsection{Proof of Lemma \ref{lem: zero}} We start with   cleaning up   the expression of $ \SZ (\frm) $ in \eqref{12eq: 0-frequency}--\eqref{12eq: Mellin}. By the change of variable $x \ra x/y$ in \eqref{12eq: Mellin}, 
\begin{align*}
	\widetilde {B}_{it} (y; s) = \|y\|_{\infty}^{-s} \widetilde {B}_{it} (  s), 
\end{align*}
where $\widetilde {B}_{it} (  s) = \widetilde {B}_{it} ( 1; s)$. 
Then the factor $ |\RN (c)|^{1 - 2v \pm 2\delta}  / |\RN(m)|^{\frac 1 2 -v \pm \delta} $ is extracted from $\widetilde {B}_{ it} \big(m/c^2; \tfrac 1 2 -v \pm \delta \big)$. The resulting $(c)$-sum may be evaluated by the Ramanujan identity:
\begin{align}\label{12eq: sum of (c)}
	\sum_{(c)   \subset \shskip \frO  }  \frac { S (m , 0; c)  } { |\RN (c  )|^{1 + 2 v} } = \frac { \tau_v (\frm) } { \RN(\frm)^v \zeta (1+2v) }, 
\end{align}
 due to \eqref{2eq: Ramanujan} and $\frm = m \frD$ (so that $\RN (\frm) = |d_F \RN (m) |$).  The two $ \zeta (1+2v) $ in  \eqref{12eq: 0-frequency} and \eqref{12eq: sum of (c)} cancel, so there is now only a simple pole at $v = 0$. By Lemma \ref{lem: Mellin}, the Mellin integral
 \begin{align}
 	\widetilde {B}_{ it} \big( \tfrac 1 2 -v \pm \delta \big) = \frac { \gamma (\tfrac 1 2 -v \pm \delta, t) } {\gamma (\tfrac 1 2 + v \mp \delta, t)} . 
 \end{align}  Moreover, $c_2 = c_1 / 2 \sqrt{|d_F|}$ (see \eqref{3eq: constants, Q} and \eqref{3eq: constants, C}).  Thus $ \SZ (\frm) $ is simplified into 
  
 \begin{align}\label{12eq: 0-frequency, 2}
 	\SZ  (\frm) =  c_1 \frac {\tau (\frm)} {\sqrt{\RN(\frm)}}   \int_{-\infty}^{\infty}    k (t) \cdot \frac { 1  } {    \pi i    } \int_{ (\frac 1 3) }    G (v, t)^2      Z(\frm; v, t) \frac {\nd v} {v}  \nd \shskip \mu (t) ,
 \end{align}
 where
 \begin{align}\label{12eq: Z '}
 	Z(\frm; v, t) = \lim_{\delta \ra 0}	\sum_{\pm} {\zeta_F (1\pm 2 \delta)   } \frac{|   d_F  | ^{ \pm 2 \delta}} {\RN (\frm)^{\pm \delta}} \frac { \gamma (\tfrac 1 2 -v \pm \delta, t) } {\gamma (\tfrac 1 2 + v \mp \delta, t)} . 
 \end{align}  
In view of \eqref{4eq: def G} and \eqref{12eq: Z '}, it is clear that $G (v, t)^2      Z(\frm; v, t) $ is   {\it even} in the $v$-variable, and therefore the $v$-integral in \eqref{12eq: 0-frequency, 2} is  equal to  exactly its value at $v = 0$ (to see this, apply $v \ra - v$ to half of the integral). Consequently, 
\begin{align}\label{12eq: 0-frequency, 3}
	\SZ  (\frm) =  c_1 \frac {\tau (\frm)} {\sqrt{\RN(\frm)}}   \int_{-\infty}^{\infty}    k (t)     Z(\frm; 0, t)   \nd \shskip \mu (t).
\end{align}
We have 
\begin{align}\label{12eq: Z(0)}
	Z (\frm; 0, t) = 2   \gamma_0 + \gamma_1 \lp \log \frac {|d_F|^2} {\RN(\frm)} + 2 \psi (t) \rp , 
\end{align}
for $\gamma_0$, $\gamma_1$, and $\psi (t)$ as in \eqref{2eq: zeta (s), s=1} and \eqref{4eq: defn of psi}. By \eqref{4eq: asymp of psi},  \eqref{12eq: 0-frequency, 3}, and \eqref{12eq: Z(0)}, we can conclude the proof with the same arguments for the diagonal term $\SD_2 (\frm)$.

\subsection{Proof of Lemma \ref{lem; bound for dual} for $F = \BQ$}

In this subsection, let $c$, $d$, $m$, and $n$    be positive integers.

It follows from $m \leqslant T^{2-\vepsilon}$ and $c^2 \Lt m R / T^2$ that  $   {n R} / {c^2}  \geqslant  T^{\vepsilon} $, so  Lemma \ref{lem: Hankel} (1) yields 
\begin{align*}
	\widetilde \varww_0 \lp \pm \frac {n R} {c^2}   ,  \frac {m R}   {c^2}   \rp =   \frac{ \sqrt{c} MT^{1+\vepsilon}  } {\sqrt[4]{n R}}   
	\Psi^{\pm} \big( \sqrt{ n/m } , \sqrt{  m R } / c \big)  
	+ O \big(T^{-A} \big) . 
\end{align*}
Recall that we defined $ \Psi^{+}  ( x , \varDelta  ) = 0   $  unless $ \varDelta > M^{1-\vepsilon} / T $ (due to Lemma \ref{lem: x > MT}).
Moreover, by  $m \shskip \leqslant T^{2  - \vepsilon}$ and  $R \shskip \leqslant T^{2  +\vepsilon}$, one may adjust $\vepsilon$ so that $ {\sqrt{m R}} / {c} \leqslant T^{2-\vepsilon}$. 
By the formula  for the Ramanujan sum $S (m \pm n, 0; c)$ in \eqref{2eq: Ramanujan} and the estimates for the $\Psi^{\pm}$-integrals in \S \ref{sec: estimates for Phi}, in particular Proposition \ref{lem: bounds for Phi, R} (1) and \ref{prop: small Phi(1)} (1), we infer  that, up to a negligibly small error,  $\sqrt{R} \, \widetilde{\SO}_2 (m;  R )$ is bounded by the sum of 
\begin{align}\label{12eq: sum -}
	\widetilde{\SO}{}^{-} (m  ) =\frac {M T^{1+\vepsilon}   } { \sqrt{m} \sqrt[4]{R}  }\sum_{0< n \shskip < m  /M^{2-\vepsilon} } \frac {\tau(n)} { \sqrt[4]{n} } \sum_{d |  m+n }  \sqrt{d}  \sum_{ c d \shskip \Lt {\sqrt{mR}}/  T} \frac{|\mu (c )|  } {\sqrt c  } , 
\end{align}
and
\begin{align}\label{12eq: sum +}
	\widetilde{\SO}{}^{+} (m  ) = \frac {M T^{1+\vepsilon}   } {   \sqrt[4]{m  R}   } \sum_{ 0 < |l| < m  /M^{2-\vepsilon} } \frac {\tau(m+l)} {\sqrt{| l   |}} \sum_{d |  l }  \sqrt{d}  \sum_{ c d \shskip <   {\sqrt{mR} } / {  M^{1-\vepsilon} T}} \frac{|\mu (c )| } {\sqrt{c}  } ,
\end{align}
with $l = n-m$.  
A critical point is that the diagonal term with $n = m$ ($l=0$) is removed from the second sum $\widetilde{\SO}{}^{+} (m)$ because it is negligibly small by Proposition \ref{prop: small Phi(1)} (1). Finally, if $ m \leqslant M^{2-\vepsilon}$ then $\widetilde{\SO}{}^{-} (m  )$ and $\widetilde{\SO}{}^{+} (m  )$ vanish since the $n$-sum and $l$-sum have no terms, and if otherwise we have estimates 
\begin{align*}%\label{12eq: sum -}
	\widetilde{\SO}{}^{-}  (m  ) \Lt \frac {M T^{1/2+\vepsilon}   } {  \sqrt[4]{m} } \sum_{0 < n \shskip < m  /M^{2-\vepsilon} } \frac {\tau(n) \tau (m+n)} { \sqrt[4]{n} } \Lt \frac {\sqrt{m } T^{1/2+\vepsilon} } {\sqrt{M}} ,
\end{align*}
\begin{align*}%\label{12eq: sum +}
	\widetilde{\SO}{}^{+}  (m ) \Lt {\sqrt{M} T^{1/2+\vepsilon}   }  \sum_{ 0 < |l| < m  /M^{2-\vepsilon} } \frac { \tau(l) \tau(m+l)} {\sqrt{| l |}}   \Lt  \frac { \sqrt{m} T^{1/2+\vepsilon}} {\sqrt{M} },  
\end{align*} 
as desired.

\subsection{Proof of Lemma \ref{lem; bound for dual} for $F = \BQ (\sqrt{d_F})$}\label{sec: off, C}

For the case $F = \BQ (\sqrt{d_F})$ we use Lemma \ref{lem: Hankel} (2), Proposition \ref{lem: bounds for Phi, R} (2) and \ref{prop: small Phi(1)} (2). Let $z = x+i y = \sqrt{n/m}$ ($m, n \in \frO' \smallsetminus \{0\}$). We partition the region  $ |x| < 1 + \rho   $ and $|y| < \rho$ in Proposition \ref{lem: bounds for Phi, R} (2) ($\rho = M^{\vepsilon}/M$) according to the $x$-coordinate as follows:
\begin{align*}
	|x| \Lt \rho, \qquad \delta  < |x| \leqslant 2 \delta , \qquad | |x| - 1| \Lt \rho, \qquad \delta  < 1 - |x| \leqslant 2 \delta, 
\end{align*}
for dyadic $\delta$ of the form $2^{-j}$ ($j = 2, 3, ...$) with $\rho \Lt \delta < 1/2$. In view of Lemma \ref{lem: square root},  the problem is reduced to proving that the following four sums have bound as in \eqref{12eq: bound for O, C}: 
\begin{align}\label{12eq: -}
	\widetilde{\SO}{}^{-} (m  ) & = \frac {M T^{2+\vepsilon}   } {  |m| \sqrt{  R}   } \sum_{ 0 < |n| < |m|  /M^{2-\vepsilon} } \frac {\tau(n \frD)} { {\sqrt{|n|}}} \SR (m-n, m) , \\
	\label{12eq: O-, 2}
	\widetilde{\SO}{}^{-}_{\delta} (m  ) & = \frac {M T^{2+\vepsilon}   } {  |m| \delta \sqrt{ |m|  R}   } \mathop{\sum_{  |\mathrm{Re}(n/m)| \asymp \delta^2 }}_{ |\mathrm{Im}(n/m)| < \delta   / M^{1-\vepsilon}  }   {\tau(n \frD)}   \SR (m-n, m) , \\
	\label{12eq: O+}
	\widetilde{\SO}{}^{+} (m  ) & = \frac {M T^{2+\vepsilon}   } {   \sqrt{|m|  R}   } \sum_{ 0 < |l| < |m|  /M^{1-\vepsilon} } \frac {\tau((m+l) \frD )} { {| l   |}} \SR (l, m) , \\
	\label{12eq: O+, 2}
	\widetilde{\SO}{}^{+}_{\delta} (m  ) &  = \frac {M T^{2+\vepsilon}   } {  |m| \delta \sqrt{ |m| R}   } \mathop{\sum_{  |\mathrm{Re}(l/m)| \asymp \delta }}_{ |\mathrm{Im}(l/m)| < M^{\vepsilon} / M }   {\tau((m+l) \frD)}   \SR (l, m) ,
\end{align}
where
\begin{align}\label{12eq: Ramanujan}
\SR (l, m) =	\sum_{ \frd |  l \frD }   \sqrt{\RN (\frd)}  \sum_{ \RN (\frc \frd) \shskip \Lt  { {|m|R} } / {   T^2 }} \frac{|\mu ( \frc )| } { \sqrt{\RN (\frc)}  } .
\end{align}
It is clear that 
\begin{align*}
\SR (l, m) = O \bigg(  \frac {\tau (l \frD) \sqrt{|m| R} } {T} \bigg),
\end{align*}
  therefore
\begin{align*}
	\widetilde{\SO}{}^{-} (m  ) \Lt \frac {M T^{1+\vepsilon}   } {   \sqrt{|m|}   } \sum_{ 0 < |n| < |m|  /M^{2-\vepsilon} } \frac {\tau(n \frD) \tau((m-n) \frD)} { {\sqrt{|n|}}} \Lt \frac {|m| T^{1+\vepsilon}} { M^2 } , 
\end{align*}
\begin{align*} 
	\widetilde{\SO}{}^{-}_{\delta} (m  ) & \Lt \frac {M T^{1+\vepsilon}   } {   {|m|} \delta   } \mathop{\sum_{  |\mathrm{Re}(n/m)| \asymp \delta^2 }}_{ |\mathrm{Im}(n/m)| < \delta   / M^{1-\vepsilon} }   {\tau(n \frD) \tau((m-n) \frD) } \\
	& \Lt  \frac {M T^{1+\vepsilon}   } {   {|m|} \delta   } \mathop{\sum_{  |\mathrm{Re}(n/m)| \Lt \delta^2 }}_{ |\mathrm{Im}(n/m)| < \delta   / M^{1-\vepsilon} } 1  ,
\end{align*}
and similarly
\begin{align*} 
	\widetilde{\SO}{}^{+} (m  ) \Lt {M T^{1+\vepsilon}   }   \sum_{ 0 < |l| < |m|  /M^{1-\vepsilon} } \frac {\tau (l \frD) \tau((m+l)\frD)} { {| l   |}}  \Lt |m| T^{1+\vepsilon}, 
\end{align*}
\begin{align*}
	\widetilde{\SO}{}^{+}_{\delta} (m  ) & \Lt    \frac {M T^{1+\vepsilon}   } {|m| \delta }   \mathop{\sum_{  |\mathrm{Re}(l/m)| \asymp \delta }}_{ |\mathrm{Im}(l/m)| < M^{\vepsilon} \hskip -1pt / M }    {\tau (l \frD) \tau((m+l)\frD)} \\
	& \Lt \frac {M T^{1+\vepsilon}   } {|m| \delta}  \mathop{\sum_{  |\mathrm{Re}(l/m)| \Lt \delta }}_{ |\mathrm{Im}(l/m)| < M^{\vepsilon} \hskip -1pt / M }  1 .   
\end{align*}

The final estimation for  $\widetilde{\SO}{}^{\pm}_{\delta} (m  )$ can be done by the next lemma.  

\begin{lem}\label{lem: count lattice points}
Let $m \in \frO'$.	For  $ Q \Lt P $ define  the rectangle $\RR (P, Q) = \big\{ x + iy : |x| < P, |y| < Q \big\}$.  The number of points in $ m^{-1} \frO' \cap \RR (P, Q) $ has bound $O \lp (|m|P+1) (|m|Q+1) \rp $. 
\end{lem}

\begin{proof}
	Firstly, it is clear that $ m \cdot \RR (P, Q)$ is contained in a parallelogram of the form $\RR_a (|m|P, |m|Q ) =  \big\{ x+iy : |x| \Lt |m| P, |y - a x| \Lt |m| Q \big\}$. Exchanging $ x   \leftrightarrow  y$ if necessary, one may assume that $|a| \leqslant 1$. Secondly, $\frO'$ is contained in a certain rectangular lattice spanned by a real scalar and an imaginary scalar. By rescaling, it is reduced to counting the integral lattice points in $\RR_a (|m|P, |m|Q )$, which can be  done very easily. 
\end{proof}

It follows from Lemma \ref{lem: count lattice points}, along with  $M^{\vepsilon}/M \Lt \delta < 1/2$, that 
\begin{align*}
\widetilde{\SO}{}^{\pm}_{\delta} (m  ) \Lt	|m|  T^{1+\vepsilon}      + \frac {M^2 T^{1+\vepsilon}   } {|m|} . 
\end{align*}

\section{Moments without Twist and Smooth Weight}

In this section, we use the unsmoothing technique  in \S \ref{sec: choice of h} to prove Corollary \ref{cor: unsmooth}. 

By the proof of  Theorem \ref{thm: moment} in the previous sections, for $ T^{\vepsilon} \leqslant M \leqslant T^{1-\vepsilon} $ we have
\begin{align}\label{13eq: M+E, 1}
	\SM_1 ( 1 ) + \SE_1 (1) = 4 \sqrt{\pi} c_1   M T^{N}    + O_{\vepsilon}  \big( M^3  / T^{2-N} \big) ,
\end{align} 
and
\begin{align}\label{13eq: M+E, 2.1}
	 \SM_2 ( 1 ) + \SE_2 (1) =  8 \sqrt{\pi}  c_1    M T          (       \log  T   + \gamma_0 '    )   +     O_{\vepsilon}    \big(   {M^3 } \log T / {T}  \big), 
\end{align}
if $F = \BQ$, and
\begin{align}\label{13eq: M+E, 2.2}
	\SM_2 ( 1 ) + \SE_2 (1) = 8 \sqrt{\pi} c_1     M T^2         (  2 \gamma_{1}   \log  T   + \gamma_0 '    )   +     O_{\vepsilon}    \big( M^2 T^{1+\vepsilon}    \big) , 
\end{align}
if $F = \BQ (\sqrt{d_F})$.\footnote{For the case $F = \BQ$, the reader may compare our formulae with those in \cite[Proposition 1]{SH-Liu-Maass}.}
It follows that 
\begin{align}
	\SM_q ( 1 ) + \SE_q (1) = O_{\vepsilon} \big(M T^{N + \vepsilon}\big)
\end{align}
for any $T^{\vepsilon} \leqslant M \leqslant T^{1-\vepsilon}$.

It is known that $ L \big(\frac 1 2 , f\big) $ is non-negative by \cite{Guo-Positive}.  Applying Lemma \ref{lem: unsmooth, 1} and \ref{lem: unsmooth, 2} (with $\lambda = N+\vepsilon$ and $a_f = L \big(\frac 1 2 , f\big)^q $) and   the averaging process to \eqref{13eq: M+E, 1}--\eqref{13eq: M+E, 2.2}, we infer that 
\begin{align*}%\label{13eq: N, 1}
	\SN_1 (T, H) =  4   c_1 \int_{\,T-H}^{T+H}  K^N \nd K    + O_{\vepsilon}  \lp M T^{N+\vepsilon}  \rp, 
\end{align*}
and
\begin{align*}%\label{13eq: N, 2.1}
	\SN_2  (T, H) = 8 c_1 \int_{\,T-H}^{T+H}  K      (       \log  K   + \gamma_0 '    ) \nd K  
	+   O_{\vepsilon}  \big( M T^{1+\vepsilon}   \big) ,
\end{align*}  
if $F = \BQ$, and 
\begin{align*}%\label{13eq: N, 2.2}
	\SN_2  (T, H) = 8 c_1 \int_{\,T-H}^{T+H} K^2         (  2 \gamma_{1}   \log  K   + \gamma_0 '    ) \nd K + O_{\vepsilon} \big(M T^{2+\vepsilon}\big), 
\end{align*} 
if $F = \BQ (\sqrt{d_F})$. Then Corollary \ref{cor: unsmooth} follows on choosing $M = T^{\vepsilon}$. % in \eqref{13eq: N, 1}, \eqref{13eq: N, 2.2}, and $M = T^{\vepsilon} + H^{2/5 } /T^{1/5-\vepsilon}   $ in \eqref{13eq: N, 2.1}. 

Finally, we remark that the arguments for $\SE_q (\frm)$ in \S \ref{sec: 1st moment} and \S\ref{sec: 2nd moment} may be easily employed here to show that,  except when $ T^{ \frac 5  7} < H < T^{ \frac 7  8 + \vepsilon} $ for $q = 2$ and $F = \BQ (\sqrt{d_F})$,  the Eisenstein contribution  in $\SN_q  (T, H)$ is $O \big(T^{N + \vepsilon} \big)$ %$ O \big(H T^{N\valpha_q + \vepsilon}\big) = O \big(T^{N + \vepsilon} \big) $ (for $H = T^{\beta}$) 
so that it may be removed from the asymptotic formulae in Corollary \ref{cor: unsmooth}.

%\section{Non-vanishing} 

%\def\cprime{$'$}


\begin{thebibliography}{EMOT}
	
	\bibitem[BB]{Blomer-Brumley}
	V.~Blomer and F.~Brumley.
	\newblock On the {R}amanujan conjecture over number fields.
	\newblock {\em Ann. of Math. (2)}, 174(1):581--605, 2011.
	
	\bibitem[BF1]{BF-prime-power}
	O.~Balkanova and D.~Frolenkov.
	\newblock Non-vanishing of automorphic {$L$}-functions of prime power level.
	\newblock {\em Monatsh. Math.}, 185(1):17--41, 2018.
	
	\bibitem[BF2]{BF-Moments}
	O.~Balkanova and D.~Frolenkov.
	\newblock Moments of {$L$}-functions and {L}iouville-{G}reen method.
	\newblock {\em J. Eur. Math. Soc. (JEMS)}, 23(4):1333--1380, 2021.
	
	\bibitem[BHS]{BHS-Maass}
	O.~Balkanova, B.~Huang, and A.~S\"{o}dergren.
	\newblock Non-vanishing of {M}aass form {$L$}-functions at the central point.
	\newblock {\em Proc. Amer. Math. Soc.}, 149(2):509--523, 2021.
	
	\bibitem[Blo]{Blomer}
	V.~Blomer.
	\newblock Subconvexity for twisted {$L$}-functions on {${\rm GL}(3)$}.
	\newblock {\em Amer. J. Math.}, 134(5):1385--1421, 2012.
	
	\bibitem[BM]{B-Mo}
	R.~W. Bruggeman and Y.~Motohashi.
	\newblock Sum formula for {K}loosterman sums and fourth moment of the
	{D}edekind zeta-function over the {G}aussian number field.
	\newblock {\em Funct. Approx. Comment. Math.}, 31:23--92, 2003.
	
	\bibitem[Bou]{Bourgain}
	J.~Bourgain.
	\newblock Decoupling, exponential sums and the {R}iemann zeta function.
	\newblock {\em J. Amer. Math. Soc.}, 30(1):205--224, 2017.
	
	\bibitem[BW]{BW-Riemann-2}
	J.~Bourgain and N.~Watt.
	\newblock Decoupling for perturbed cones and the mean square of {$|\zeta (\frac
		12+it)|$}.
	\newblock {\em Int. Math. Res. Not. IMRN}, (17):5219--5296, 2018.
	
	
	
	\bibitem[Dja]{Djankovic-Gamma1}
	G.~Djankovi\'{c}.
	\newblock Nonvanishing of the family of {$\Gamma_1(q)$}-automorphic
	{$L$}-functions at the central point.
	\newblock {\em Int. J. Number Theory}, 7(6):1423--1439, 2011.
	
	\bibitem[Duk]{Duke-1995}
	W.~Duke.
	\newblock The critical order of vanishing of automorphic {$L$}-functions with
	large level.
	\newblock {\em Invent. Math.}, 119(1):165--174, 1995.
	
	\bibitem[EGM]{EGM}
	J.~Elstrodt, F.~Grunewald, and J.~Mennicke.
	\newblock {\em Groups {A}cting on {H}yperbolic {S}pace}.
	\newblock Springer Monographs in Mathematics. Springer-Verlag, Berlin, 1998.
	
	\bibitem[EMOT]{ET-II}
	A.~Erd{\'e}lyi, W.~Magnus, F.~Oberhettinger, and F.~G. Tricomi.
	\newblock {\em Tables of {I}ntegral {T}ransforms. {V}ol. {II}}.
	\newblock McGraw-Hill Book Company, Inc., New York-Toronto-London, 1954.
	\newblock Based, in part, on notes left by Harry Bateman.
	
	\bibitem[GN]{Gale-Nikaido}
	D.~Gale and H.~Nikaid\^{o}.
	\newblock The {J}acobian matrix and global univalence of mappings.
	\newblock {\em Math. Ann.}, 159:81--93, 1965.
	
	\bibitem[Guo]{Guo-Positive}
	J.~Guo.
	\newblock On the positivity of the central critical values of automorphic
	{$L$}-functions for {${\rm GL}(2)$}.
	\newblock {\em Duke Math. J.}, 83(1):157--190, 1996.
	
	\bibitem[HB1]{H-B-Hybrid}
	D.~R. Heath-Brown.
	\newblock Hybrid bounds for {D}irichlet {$L$}-functions.
	\newblock {\em Invent. Math.}, 47(2):149--170, 1978.
	
	\bibitem[HB2]{Heath-Brown-Weyl}
	D.~R. Heath-Brown.
	\newblock The growth rate of the {D}edekind zeta-function on the critical line.
	\newblock {\em Acta Arith.}, 49(4):323--339, 1988.
	
	\bibitem[Hou]{Hough-Zero-Density}
	B.~Hough.
	\newblock Zero-density estimate for modular form {$L$}-functions in weight
	aspect.
	\newblock {\em Acta Arith.}, 154(2):187--216, 2012.
	
	\bibitem[IJ]{Iviv-Jutila-Moments}
	A.~Ivi\'{c} and M.~Jutila.
	\newblock On the moments of {H}ecke series at central points. {II}.
	\newblock {\em Funct. Approx. Comment. Math.}, 31:93--108, 2003.
	
	\bibitem[IK]{IK}
	H.~Iwaniec and E.~Kowalski.
	\newblock {\em Analytic {N}umber {T}heory},  {American
		Mathematical Society Colloquium Publications, Vol. 53}.
	\newblock American Mathematical Society, Providence, RI, 2004.
	
	\bibitem[IM]{IM-4th-Moment}
	A.~Ivi\'{c} and Y.~Motohashi.
	\newblock On the fourth power moment of the {R}iemann zeta-function.
	\newblock {\em J. Number Theory}, 51(1):16--45, 1995.
	
	\bibitem[IS]{IS-Siegel}
	H.~Iwaniec and P.~Sarnak.
	\newblock The non-vanishing of central values of automorphic {$L$}-functions
	and {L}andau-{S}iegel zeros.
	\newblock {\em Israel J. Math.}, 120(part A):155--177, 2000.
	
	\bibitem[Ivi]{Ivic-Riemann-4}
	Aleksandar Ivi\'{c}.
	\newblock On mean value results for the {R}iemann zeta-function in short
	intervals.
	\newblock {\em Hardy-Ramanujan J.}, 32:4--23, 2009.
	
	\bibitem[Job]{Jobrack-Derivative}
	M.~Jobrack.
	\newblock Non-vanishing of the derivative of {$L$}-functions at the central
	point.
	\newblock {\em J. Number Theory}, 209:49--82, 2020.
	
	\bibitem[Joh]{Faa-di-Bruno}
	W.~P. Johnson.
	\newblock The curious history of {F}a\`a di {B}runo's formula.
	\newblock {\em Amer. Math. Monthly}, 109(3):217--234, 2002.
	
	\bibitem[KM1]{KM-Analytic-Rank}
	E.~Kowalski and P.~Michel.
	\newblock The analytic rank of {$J_0(q)$} and zeros of automorphic
	{$L$}-functions.
	\newblock {\em Duke Math. J.}, 100(3):503--542, 1999.
	
	\bibitem[KM2]{KM-Analytic-Rank-2}
	E.~Kowalski and P.~Michel.
	\newblock A lower bound for the rank of {$J_0(q)$}.
	\newblock {\em Acta Arith.}, 94(4):303--343, 2000.
	
	\bibitem[KMV]{KMV-Derivatives}
	E.~Kowalski, P.~Michel, and J.~VanderKam.
	\newblock Non-vanishing of high derivatives of automorphic {$L$}-functions at
	the center of the critical strip.
	\newblock {\em J. Reine Angew. Math.}, 526:1--34, 2000.
	
	\bibitem[Kuz1]{Kuznetsov}
	N.~V. Kuznetsov.
	\newblock {P}etersson's conjecture for cusp forms of weight zero and {L}innik's
	conjecture. {S}ums of {K}loosterman sums.
	\newblock {\em Math. Sbornik}, 39:299--342, 1981.
	
	\bibitem[Kuz2]{Kuznetsov-Motohashi-formula}
	N.~V. Kuznetsov.
	\newblock Convolution of {F}ourier coefficients of {E}isenstein-{M}aass series.
	\newblock {\em Zap. Nauchn. Sem. Leningrad. Otdel. Mat. Inst. Steklov. (LOMI)},
	129:43--84, 1983.
	
	\bibitem[Li]{XLi2011}
	X.~Li.
	\newblock Bounds for {${\rm GL}(3)\times {\rm GL}(2)$} {$L$}-functions and
	{${\rm GL}(3)$} {$L$}-functions.
	\newblock {\em Ann. of Math. (2)}, 173(1):301--336, 2011.
	
	\bibitem[Liu1]{SH-Liu-Maass}
	S.~Liu.
	\newblock Nonvanishing of central {$L$}-values of {M}aass forms.
	\newblock {\em Adv. Math.}, 332:403--437, 2018.
	
	\bibitem[Liu2]{Liu-L-Derivative}
	S.~Liu.
	\newblock On central {$L$}-derivative values of automorphic forms.
	\newblock {\em Math. Z.}, 288(3-4):1327--1359, 2018.
	
	\bibitem[LQ]{Qi-Liu-LLZ}
	S.-C. Liu. and Z.~Qi.
	\newblock Low-lying zeros of {$L$}-functions for {M}aass forms over imaginary
	quadratic fields.
	\newblock {\em Mathematika}, 66(3):777--805, 2020.
	
	\bibitem[LT]{Lau-Tsang-Mean-Square}
	Y.-K. Lau and K.-M. Tsang.
	\newblock A mean square formula for central values of twisted automorphic
	{$L$}-functions.
	\newblock {\em Acta Arith.}, 118(3):231--262, 2005.
	
	\bibitem[Luo]{Luo-Weight}
	W.~Luo.
	\newblock Nonvanishing of the central {$L$}-values with large weight.
	\newblock {\em Adv. Math.}, 285:220--234, 2015.
	
	\bibitem[M\"ul]{Muller-Dedekind-Quadratic}
	W.~M\"{u}ller.
	\newblock The mean square of the {D}edekind zeta function in quadratic number
	fields.
	\newblock {\em Math. Proc. Cambridge Philos. Soc.}, 106(3):403--417, 1989.
	
	\bibitem[Mot1]{Motohashi-JNT-Mean}
	Y.~Motohashi.
	\newblock Spectral mean values of {M}aass waveform {$L$}-functions.
	\newblock {\em J. Number Theory}, 42(3):258--284, 1992.
	
	\bibitem[Mot2]{Motohashi-Riemann}
	Y.~Motohashi.
	\newblock {\em Spectral theory of the {R}iemann zeta-function},  
	{Cambridge Tracts in Mathematics, Vol. 127}.
	\newblock Cambridge University Press, Cambridge, 1997.
	 
	
\bibitem[Qi1]{Qi-Gauss}
Z.~Qi.
\newblock Subconvexity for twisted {$L$}-functions on {$\rm{GL}_3$} over the
{G}aussian number field.
\newblock {\em Trans. Amer. Math. Soc.}, 372(12):8897--8932, 2019.



\bibitem[Qi2]{Qi-Bessel}
Z.~Qi.
\newblock Theory of fundamental {B}essel functions of high rank.
\newblock {\em Mem. Amer. Math. Soc.}, 267(1303):vii+123, 2020.

\bibitem[Qi3]{Qi-GL(3)}
Z.~Qi.
\newblock Subconvexity for {$L$}-functions on {$\mathrm{GL}_3$} over number
fields.
\newblock {\em to appear in J. Eur. Math. Soc. (JEMS)}, 2020.

\bibitem[Qi4]{Qi-BE}
Z.~Qi.
\newblock On the {F}ourier transform of regularized {B}essel functions on
complex numbers and {B}eyond {E}ndoscopy over number fields.
\newblock {\em Int. Math. Res. Not. IMRN}, (19):14445--14479, 2021.

\bibitem[Qi5]{Qi-VO}
Z.~Qi.
\newblock A {V}orono\"i--{O}ppenheim summation formula for number fields.
\newblock {\em to appear in J. Number Theory}, 2021.


	\bibitem[Rou1]{Rouymi-1}
	D.~Rouymi.
	\newblock Formules de trace et non-annulation de fonctions {$L$} automorphes au
	niveau {$ \mathfrak{p}^\nu$}.
	\newblock {\em Acta Arith.}, 147(1):1--32, 2011.
	
	\bibitem[Rou2]{Rouymi-2}
	D.~Rouymi.
	\newblock Mollification et non annulation de fonctions {$L$} automorphes en
	niveau primaire.
	\newblock {\em J. Number Theory}, 132(1):79--93, 2012.
	
	\bibitem[Tro]{Trotabas-Hilbert}
	D.~Trotabas.
	\newblock Non annulation des fonctions {$L$} des formes modulaires de {H}ilbert
	au point central.
	\newblock {\em Ann. Inst. Fourier (Grenoble)}, 61(1):187--259, 2011.
	
	\bibitem[Van]{VanderKam-Rank}
	J.~M. VanderKam.
	\newblock The rank of quotients of {$J_0(N)$}.
	\newblock {\em Duke Math. J.}, 97(3):545--577, 1999.
	
	\bibitem[Ven]{Venkatesh-BeyondEndoscopy}
	A.~Venkatesh.
	\newblock ``{B}eyond endoscopy'' and special forms on $\mathrm{GL}(2)$.
	\newblock {\em J. Reine Angew. Math.}, 577:23--80, 2004.
	
	\bibitem[Wat]{Watson}
	G.~N. Watson.
	\newblock {\em A {T}reatise on the {T}heory of {B}essel {F}unctions}.
	\newblock Cambridge University Press, Cambridge, England; The Macmillan
	Company, New York, 1944.
	
%	\bibitem[Xu]{Xu-Nonvanishing}
%	Z.~Xu.
%	\newblock Nonvanishing of automorphic {$L$}-functions at special points.
%	\newblock {\em Acta Arith.}, 162\allowbreak (4):\allowbreak 309--335, 2014.
	
	\bibitem[You]{Young-Cubic}
	M.~P. Young.
	\newblock Weyl-type hybrid subconvexity bounds for twisted {$L$}-functions and
	{H}eegner points on shrinking sets.
	\newblock {\em J. Eur. Math. Soc. (JEMS)}, 19(5):1545--1576, 2017.
	
\end{thebibliography}
\end{document}